\documentclass{elsarticle}
\usepackage{amsmath,amsfonts,amssymb,url,algorithm,algpseudocode}
\usepackage{xcolor}
\newtheorem{theorem}{Theorem}[section]
\newtheorem{cor}[theorem]{Corollary}
\newtheorem{prop}[theorem]{Proposition}
\newtheorem{lemma}[theorem]{Lemma}
\newtheorem{claim}[theorem]{Claim}
\newproof{proof}{Proof}
\newproof{poc}{Proof of Claim}

\newdefinition{remark}{Remark}
\newdefinition{defn}{Definition}
\newdefinition{ques}{Question}
\newdefinition{obs}{Observation}
\newdefinition{problem}{Problem}

\newcommand{\Aut}{\mathop{\mathrm{Aut}}}
\newcommand{\Span}{\mathop{\mathrm{Span}}}
\newcommand{\desc}{\mathop{\mathrm{desc}}}
\newcommand{\bin}{\mathop{\mathrm{bin}}}
\newcommand{\size}{\mathop{\mathrm{size}}}
\newcommand{\Inn}{\mathop{\mathrm{Inn}}}
\newcommand{\Soc}{\mathop{\mathrm{Soc}}}
\newcommand{\Out}{\mathop{\mathrm{Out}}}
\newcommand{\Comp}{\mathop{\mathrm{Comp}}}

\newcommand{\GL}{\mathop{\mathrm{GL}}\nolimits}
\newcommand{\PSL}{\mathop{\mathrm{PSL}}\nolimits}
\newcommand{\SL}{\mathop{\mathrm{SL}}\nolimits}
\newcommand{\Sz}{\mathop{\mathrm{Sz}}\nolimits}
\newcommand{\GU}{\mathop{\mathrm{GU}}\nolimits}
\newcommand{\PSU}{\mathop{\mathrm{PSU}}\nolimits}
\newcommand{\SU}{\mathop{\mathrm{SU}}\nolimits}
\newcommand{\PSp}{\mathop{\mathrm{PSp}}\nolimits}
\newcommand{\POm}{\mathop{\mathrm{P}\Omega}\nolimits}
\newcommand{\GF}{\mathop{\mathrm{GF}}\nolimits}

\newcommand{\Sp}{\mathop{\mathrm{Sp}}\nolimits}
\newcommand{\GO}{\mathop{\mathrm{GO}}\nolimits}
\newcommand{\SO}{\mathop{\mathrm{SO}}\nolimits}
\newcommand{\PSO}{\mathop{\mathrm{PSO}}\nolimits}

\DeclareMathOperator{\Mod}{mod}

\newenvironment{claimproof}[1]{\par\noindent{\em Proof of Claim.}~#1}

\newcommand{\F}{\mathbb{F}}

\renewcommand{\angle}[1]{\langle#1\rangle} 
\newcommand{\sympl}[1]{\langle #1 \rangle}

\begin{document}

\begin{frontmatter}
\title{Aspects of the commuting graph\tnoteref{t1}}
\tnotetext[t1]{To the memory of Richard Parker}
\author[1]{V. Arvind}
\ead[1]{arvind@imsc.res.in}
\author[2]{Peter J. Cameron\corref{cor}}
\ead[2]{pjc20@st-andrews.ac.uk}
\cortext[cor]{Corresponding author}
\author[3]{Xuanlong Ma}
\ead[3]{xuanlma@xsyu.edu.cn}
\author[4]{Natalia V. Maslova}
\ead[4]{butterson@mail.ru}

\address[1]{The Institute of Mathematical Sciences
    (HBNI), Chennai, India and Chennai Mathematical Institute,
    Siruseri, India}
\address[2]{School of Mathematics and Statistics, University of St Andrews,
    St Andrews, Fife KY16 9SS, UK}
\address[3]{School of Science, Xi'an Shiyou University, Xi'an 710065,
    P.R. China}
\address[4]{Krasovskii Institute of Mathematics and Mechanics UB RAS,
    Yekaterinburg, 620077, Russia, and Ural Federal University, Yekaterinburg,
    620002, Russia}

\begin{abstract}

Our purpose in this paper is twofold.
\begin{enumerate}
\item We discuss the computational problem of deciding whether a given graph
is the commuting graph of a finite group; we give a quasipolynomial algorithm,
and a polynomial algorithm for the case when the group is an extra\-special
$p$-group for $p$ an odd prime.
\item We give new results on the question of whether the commuting graph of
a given group is a cograph or a chordal graph, two classes of graphs defined
by forbidden subgraphs.
\end{enumerate}
The problems are not unrelated, since there are a number of cases where hard
computational problems on graphs are easier when restricted to special
classes of graphs; we conjecture that the recognition problem is polynomial
for cographs and chordal graphs.

\end{abstract}


\end{frontmatter}

\section{Introduction}

Graphs associated with some algebraic structures, such as the power graph of a group or semigroup, have been actively investigated (cf. \cite{AKC,Cam10,CGh, CaMas}) Of course the famous Cayley graphs have a long history. Moreover, graphs from algebraic structures
have valuable applications (cf. \cite{Kelbook}): for example, Cayley graphs as classifiers for data mining \cite{KeR}.

Let $G$ be a group and $S$ be a subset of $G$. The {\em commuting
  graph} of $G$ on $S$ is a simple graph with vertex set $S$, and two
distinct elements are adjacent in this graph if they commute in
$G$. In particular, if $S=G$, then we denote this commuting graph on
$G$ by $\Gamma(G)$. In 1955, Brauer and Fowler \cite{BF55} first
introduced the commuting graph of a group and showed that there are
only finitely many simple groups with even order such that they have a
prescribed involution centralizer. In fact, in the seminal paper
\cite{BF55}, they considered the induced subgraph of the commuting
graph on the set of all non-trivial elements of the group.  Since
then, the commuting graphs have received considerable attention and
have been well studied in the literature.  For example, Segev {\em et
  al.}~\cite{Com4,Com5,Com6} established some long\-standing
conjectures on division algebras by using combinatorial
parameters of commuting graphs.  Britnell and Gill \cite{Br17}
classified all finite quasi-simple groups $G$ whose commuting graph
$\Gamma(G)$ is perfect, where $\Gamma(G)$ is defined on the vertex set
$G\setminus Z(G)$, where $Z(G)$ is the centre of $G$. Giudici and
Parker \cite{GiP} showed that for every positive integer $d$, there is
a finite $2$-group $G$ such that the commuting graph of $G$ on the set
of all non-central elements has diameter greater than $d$; further
results are found in \cite{Cu22}.  Morgan and Parker \cite{MoP} proved
that, for a finite group $G$ with trivial centre, every connected
component of the commuting graph of $G$ on the set of all non-trivial
elements has diameter at most $10$.  Recently, Beike {\em et
  al.}~\cite{Bei} extended some of their results; some further
applications of commuting graphs in group theory can be found in
\cite{GGsurvey}.

The first part of this paper is devoted to the following algorithmic
question.

\begin{ques}\label{ques0}
  Given a simple undirected graph $\Gamma =(V,E)$, with vertex set $V$
  and edge set $E$, is it isomorphic to the commuting graph
  $\Gamma(G)$ of a group $G$? More precisely, is there a bijection
  $\varphi:V\to G$ such that a vertex pair $\{u,v\}$ is an edge in
  $\Gamma$ precisely when the groups elements
  $\varphi(u),\varphi(v)\in G$ commute.
\end{ques}

We give a quasi-polynomial algorithm to answer this question in
general, and a polynomial-time algorithm in the case of extraspecial
$p$-groups with $p$ odd.

A number of important graph classes, including perfect graphs,
cographs, chordal graphs, split graphs, and threshold graphs, can be
defined either structurally or in terms of forbidden induced
subgraphs. Forbidden subgraphs of power graphs were studied in
\cite{MCam}.  In \cite{GGsurvey}, the second author proposed the following
question.


\begin{ques}\label{ques1}{\rm (\cite[Question 14]{GGsurvey})}
For which finite groups is the commuting graph a perfect graph, or a
cograph, or a chordal graph, or a split graph, or a threshold graph?
\end{ques}

Britnell and Gill \cite{Br17} studied the finite groups whose commuting graph is perfect.  The second and third authors~\cite{MaCam} completely classified the finite groups whose commuting graph is a split graph or a threshold graph.

In Part II of the paper (Sections \ref{ten-sec}, \ref{Structure},
\ref{Simple}, and \ref{appendix}) , we address the question of when
the commuting graph is a cograph or a chordal graph. We prove that in
some sense, the class of such groups generalizes the class of finite
solvable groups (see Theorem~\ref{CommonStructure_Cograph} in
Subsection~\ref{General}). Also we determine all finite simple groups
whose commuting graph is a cograph (see Problem~\ref{Classify} in
Subsection~\ref{General} and Theorem~\ref{Simple-cog} at the end of
Section~\ref{Simple}). The results partly answer Question~\ref{ques1}.

These questions are relevant to the recognition problem
(Question~\ref{ques0}). There are extra hypotheses which could make
the problem of recognition of commuting graphs easier and yield faster
algorithms. One is to restrict the groups to a known class, such as
nilpotent groups. Another is to restrict the graphs to a class in
which computation is easier. For example, the result of the second and
third authors \cite{MaCam} describes precisely the split or threshold
graphs which are commuting graphs; these are the complete graphs,
together with the graphs consisting of an $n$-vertex complete graph
(where $n$ is odd), with $n$ pendant edges added at one of its
vertices. Clearly these are easily recognised.


We note that some works, especially on connectedness and diameter,
define the commuting graph to have vertex set $G\setminus Z(G)$ (that
is, omitting dominating vertices). This has little effect on our
results. For the recognition problem, if we are given the commuting
graph on $G\setminus Z(G)$, then the commuting graph on $G$ is found
by adding $|Z(G)|$ dominating vertices to the graph, and this number
is a divisor of $|G\setminus Z(G)|$ so there are not many
possibilities. In the second part of the paper too, as mentioned
before Proposition~\ref{cog-subg}, we can define the commuting graph
without including the centre $Z(G)$ and the results essentially remain
the same.

We conclude the introduction with a suggestion for further research,
which would link the two parts of the paper together more closely.

\begin{ques}
Suppose that $\Gamma$ is a perfect graph, a cograph or a chordal graph. Is
there a polynomial algorithm to decide whether $\Gamma$ is the commuting graph
of a group?
\end{ques}

\section{Preliminaries}\label{prelim-sec}
In this section, we will introduce some notation and recall some
definitions on groups and graphs.

Every group considered is finite. Throughout, $G$ always denotes a
finite group and $1_G$ denotes the identity element of $G$, and we
denote $1_G$ simply by $1$ if the considered group is unambiguous.
The {\em order} of the group $G$ is its cardinality, denoted
$|G|$. For an element $g\in G$, the {\em order} $o(g)$ of $g$ in $G$
is the size of the cyclic subgroup $\langle g\rangle$ generated by
$g$. The element $g$ is called an {\em involution} if $o(g)=2$. The
centre of $G$, denoted by $Z(G)$, is defined as the subgroup
\[
\{x\in G: gx=xg \hbox{ for each $g\in G$}\}.
\]
The {\em centralizer} $C_G(g)$ of $g$ in $G$ is the subgroup of
all elements that commute with $g$, that is,
\[
C_G(g)=\{x\in G: gx=xg\}.
\]
The {\em commutator} of two elements $x,y$ is denoted as $[x,y]$, that
is, $[x,y]=x^{-1}y^{-1}xy$. Note that $x$ and $y$ commute if and only
if $[x,y]=1$. The \emph{derived group}, or \emph{commutator subgroup}, $G'$
of $G$ is the subgroup generated by all commutators. The \emph{derived series}
is the series of subgroups $G^{(n)}$ of $G$ defined by $G^{0}=G$ and
$G^{(n+1)}=(G^{(n)})'$. The group $G$ is \emph{solvable} if $G^{(n)}=1$ for
some $n$.

Denote by $\mathbb{Z}_n$ the cyclic group with order
$n$.  For a prime $p$, a \emph{$p$-group} is a group whose order is a
power of $p$.  An {\em elementary abelian $p$-group} is a $p$-group
that is a direct product of cyclic groups of order $p$. A {\em Sylow
  $p$-subgroup} of $G$ is a $p$-subgroup whose index in $G$ is coprime
to $p$.

An \emph{automorphism} of $G$ is a permutation $\alpha$ of $G$ preserving
the group operation: $(gh)^\alpha=g^\alpha h^\alpha$. An \emph{inner
automorphism} is a map $\phi(g):x\mapsto g^{-1}xg$. It is readily checked
that the automorphisms of $G$ form a group $\Aut(G)$, while the
inner automorphisms form a subgroup $\Inn(G)$ of $\Aut(G)$.

A subgroup $N$ of $G$ is a \emph{normal subgroup} of $G$, denoted
$N\lhd G$, if $gN=Ng$ for all $g\in G$. The \emph{quotient} $G/N$ is
the group whose elements are the cosets of $N$ in $G$ with the (well-defined)
group operation $(Ng_1)(Ng_2)=N(g_1g_2)$. For example,
\begin{itemize}
\item the centre $Z(G)$ is a normal subgroup of $G$, and $G/Z(G)$ is isomorphic
to $\Inn(G)$;
\item the inner automorphism group $\Inn(G)$ is a normal subgroup of the
automorphism group $\Aut(G)$; the quotient is the \emph{outer automorphism
group} $\Out(G)$ of $G$;
\item the derived group $G'$ is a normal subgroup of $G$; it is the smallest
normal subgroup $N$ for which $G/N$ is abelian.
\end{itemize}

A subgroup of $G$ is \emph{characteristic} if it is mapped to itself by all
automorphisms of $G$. Since a subgroup is normal if it is mapped to itself
by all inner automorphisms, we see that each characteristic subgroup is normal.

A group $G$ is \emph{simple} if its only normal subgroups are $G$ and $\{1\}$.
In this paper, we will use the \emph{Classification of the Finite Simple Groups}
(CFSG), which will be explained when it is used in
Section~\ref{cfsg-sec}. Here we note one
consequence of CFSG, the proof of the \emph{Schreier conjecture}:

\begin{theorem}\label{schreier}
If $G$ is a non-abelian finite simple group, then $\Out(G)$ is solvable.
\end{theorem}

Each minimal (non-trivial) normal subgroup of a finite group is a direct
product of isomorphic simple groups~\cite[Theorem 4.3A(iii)]{DM}.
The \emph{socle} of $G$, denoted $\Soc(G)$, is the product of the minimal
normal subgroups of $G$; it is a direct product of simple groups (not
necessarily isomorphic). If the socle is simple, we say that $G$ is
\emph{almost simple}.

A {\em normal series} for
a group $G$ is a collection of normal subgroups $N_i\lhd G$ with $0\le i\le
r$ such that
\[
\{1\}=N_0\subset N_1\subset N_2\subset\cdots \subset N_r=G.
\]
This series is a {\em central series} for $G$ if
$N_i/N_{i-1}\subseteq Z(G/N_{i-1})$ for $1\le i\le r$. A group $G$ is
\emph{nilpotent} if it has a central series. If $G$ is nilpotent, and we
build the \emph{upper central series} by taking
$N_i/N_{i-1}=Z(G/N_{i-1})$, the number of steps in the series is the
\emph{nilpotency class} of $G$.

\begin{theorem}\rm{\cite[Theorem 1.26]{isaacs}}
If $G$ is a finite group, the following conditions are equivalent:
\begin{enumerate}
\item $G$ is nilpotent.
\item Every Sylow subgroup of $G$ is normal in $G$.
\item $G$ is the direct product of its Sylow subgroups.
\end{enumerate}
\end{theorem}

A \emph{central product} $H\circ K$ of two groups $H$ and $K$ with respect to an
isomorphism $\theta$ from a central subgroup $Z$ of $H$ to a central subgroup of
$K$ is defined to be $$(H\times K)/\{(z,z\theta):z\in Z\}.$$

By $F(G)$ we will denote the {\it Fitting subgroup} of $G$, i.e.\ the largest normal nilpotent subgroup of $G$.

A \emph{component} of $G$ is a subgroup $H$ of $G$ which is \emph{quasisimple}
(that is, $H=H'$ and $H/Z(H)$ is simple) and subnormal in $G$ (that is, a term
in some normal series of $G$). Let $\Comp(G)$ be the set of components of $G$.
Then the subgroup $E(G)$ is the subgroup generated by $\Comp(G)$, and the
\emph{generalized Fitting subgroup} $F^*(G)$ of $G$ is the central product
\begin{equation}\label{F(G)E(G)} F^*(G)=F(G) \circ E(G), \end{equation}
see \cite[31.12]{Asch86}.

Given a positive integer $n$, we use $D_{2n}$ to denote the {\em
  dihedral group} with order $2n$. A presentation of $D_{2n}$ is given
by
\begin{equation}\label{d2n}
D_{2n}=\langle a,b: a^n=b^2=1, bab=a^{-1}\rangle.
\end{equation}
Note that $D_{2n}$ is an abelian group if and only if $n\in \{1,2\}$, and for all $1\le i \le n$,
\begin{equation}\label{d2n-1}
o(a^ib)=2,~~
D_{2n}=\langle a\rangle \cup \{b,ab,a^2b,\ldots,a^{n-1}b\}.
\end{equation}

Suppose that $A$ is an abelian group. The {\em generalized dihedral
  group} corresponding to $A$, denoted by $D(A)$, is the external
semidirect product $A\rtimes T$ of $A$ with the cyclic group
$T=\langle t \rangle$ of order two, with the non-identity element $t$
acting as the inverse map on $A$, that is,
\[
x^t=x^{-1} \textrm{ for all } x\in A.
\]
Viewing this as an internal semidirect
product, $A$ is an abelian normal subgroup of $D(A)$ of index two. A
presentation for $D(A)$ is as follows:
\begin{equation}\label{gdg}
D(A)=\langle A,t: o(t)=2, tgt=g^{-1}\hbox{ for each $g\in A$}\rangle.
\end{equation}

\begin{obs}\label{gdg-ob}
Let $D(A)$ be the generalized dihedral group as presented
in \eqref{gdg}. Then we have the following.
\begin{itemize}
  \item[{\rm (i)}] Every element of $D(A)\setminus A$ is an
    involution;
  \item[{\rm (ii)}] $D(A)$ is abelian if and only if $A$ is an
    elementary abelian $2$-group, if and only if $D(A)$ is an
    elementary abelian $2$-group;
  \item[{\rm (iii)}] If $A$ is cyclic of order $n$, then $D(A)\cong
    D(\mathbb{Z}_n)\cong D_{2n}$.
\end{itemize}
\end{obs}
\begin{proof}
  As $A$ is a normal subgroup of $D(A)$ of index two, every element
  $g\in D(A)\setminus A$ is in the coset $At$ and is of the form
  $g=at$. Hence, $g^2=atat=aa^t=aa^{-1}=1$ which means $g$ is an
  involution. Furthermore, $D(A)$ is commutative if and only if for
  all $a\in A$, $at=ta$. That is, $tat=a^{-1}=a$ which implies $a^2=1$
  for all $a\in A$, implying that $A$ is an elementary abelian
  $2$-group which, in turn, implies that $D(A)$ is an elementary
  abelian $2$-group. The last part follows directly from the definition of
  dihedral groups. \qed
\end{proof}

Observation~\ref{gdg-ob}(ii) means that $D(A)$ is non-abelian if and
only if $A$ has an element of order at least $3$.  Moreover,
Observation~\ref{gdg-ob}(iii) implies that the dihedral groups are
special cases of generalized dihedral groups. As an example, one can
check that $D(\mathbb{Z}_2\times \mathbb{Z}_6)\cong \mathbb{Z}_2\times
D_{12}$.

Given a positive integer $m$ at least $2$, Johnson
\cite[pp. 44--45]{Jon} introduced the {\em generalized quaternion
  group} $Q_{4m}$ of order $4m$, that is,
\begin{equation}\label{q4m}
Q_{4m}=\langle x,y: x^m=y^2, x^{2m}=1, y^{-1}xy=x^{-1}\rangle.
\end{equation}

\begin{obs}\label{obs2}
For any $1\le i \le 2m$, $o(x^iy)=4$ and $Z(Q_{4m})=\{1,x^m\}$.
\end{obs}

\begin{proof}
  By definition, $x^{2m}=y^4=1$, which implies that $o(y)=4$, as there
  are no other relations. Furthermore, as $xy=yx^{-1}$ we can write
  $x^iy=yx^{-i}$ for any positive integer $i$. It follows that
  $x^iyx^iy=yx^{-i}x^iy=y^2$ for all $i$. Hence, $(x^iy)^4=y^4=1$ for
  all $i$, implying that $o(x^iy)=4$ for $1\le i\le 2m$. We can
  express every element in $Q_{4m}$ as $x^iy^j$ for $0\le i\le 2m-1$
  and $0\le j\le 3$. Now, the element $x^iy^j$ is in the centre
  $Z(Q_{4m})$ if and only if it commutes with both $y$ and $x$.
  Clearly, $yx^iy^j=x^iy^jy$ if and only if $yx^i=x^iy$ if and only if
  $yx^i=yx^{-i}$, which is equivalent to $x^{2i}=1$. This forces
  $i\in\{0,m\}$. Thus, $x^m$ is in the centre $Z(Q_{4m})$. Suppose
  $j\in\{0,2\}$. If $j\in\{0,2\}$ then, as $i\in\{0,m\}$, and
  $x^m=y^2$ by definition of the group $Q_{4m}$, it follows that
  $x^iy^j\in \{1,x^m\}$ which is contained in the centre
  $Z(Q_{4m})$. Consider the case $j\in\{1,3\}$. Then
  $x^my^jx=x^{m+1}y^j$ if and only if $y^jx=xy^j$, which is equivalent
  to $y^jx=y^jx^{-1}$ as $xy^j=y^jx^{-1}$ since $j$ is odd.  But that
  implies $x^2=1$ which is a contradiction as $o(x)=2m>2$. Hence
  $Z(Q_{4m})=\{1,x^m\}$. \qed
\end{proof}


The group of all permutations of $\{1,2,\ldots,n\}$ is the \emph{symmetric
group} $S_n$. For $n\ge2$, its derived group is the \emph{alternating group}
$A_n$ consisting of all even permutations. The alternating group is simple
for $n\ge5$.

A group $G$ is called a {\em Frobenius group} if there is a proper
subgroup $H$ of $G$ such that $H \cap H^g=1$ whenever
$g \in G\setminus H$. Let
$$ K=\{1_G\} \cup (G \setminus ( \cup_{g \in G} H^g))$$ be the {\em
  Frobenius kernel} of $G$. It is well-known (see, for
example,~\cite[35.24 and~35.25]{Asch86}) that $K \trianglelefteq G$,
$G=K \rtimes H$, $C_G(h)\le H$ for each $h \in H$, and $C_G(k)\le K$
for each $k \in K\setminus \{1\}$.  Moreover, by the Thompson theorem on finite
groups with fixed-point-free automorphisms of prime
order~\cite[Theorem~1]{Thompson}, $K$ is nilpotent.  Frobenius groups
are also characterised as transitive non-regular permutation groups in
which non-identity elements fix at most one point.

\medskip

The \emph{Frattini subgroup} $\Phi(G)$ of a finite group $G$ is the
intersection of all maximal proper subgroups of $G$. A $p$-group $G$
is an \emph{extraspecial} group $G$ if its centre $Z(G)$ is of order $p$
and coincides with its derived subgroup $G'$ and the Frattini subgroup
$\Phi(G)$ \cite[Chapter 4, pp.\ 123]{isaacs}. Furthermore, there are
exactly two non-isomorphic extraspecial groups of order $p^{2r+1}$,
for each $r\ge 1$ and each prime $p$.

Every graph considered in this paper is simple: that is, it is
undirected, with no loops and no multiple edges.  Given a graph
$\Gamma$, the vertex set and edge set of $\Gamma$ are denoted by
$V(\Gamma)$ and $E(\Gamma)$, respectively. We often write $\Gamma$ as
a pair $\Gamma=(V,E)$ where $V=V(\Gamma)$ and $E=E(\Gamma)$.  For a
vertex subset $U\subset V(\Gamma)$ the {\em induced subgraph}
$\Gamma[U]$ has vertex set $U$ and edge set
$\{\{u,v\}\in E(\Gamma)\mid u,v\in U\}$. In other words, the subgraph
\emph{induced} by $U$ has vertex set $U$ and its edge set consists of
all edges of $\Gamma$ between vertices in $U$.

Let $\mathcal{F}$ be a collection of finite graphs. If no induced
subgraph of $\Gamma$ is isomorphic to a graph in $\mathcal{F}$, then
$\Gamma$ is called {\em $\mathcal{F}$-free}; and if
$\mathcal{F}=\{\Delta\}$, for some graph $\Delta$, we say that
$\Gamma$ is \emph{$\Delta$-free}. Sometimes we say instead that
$\Gamma$ forbids $\mathcal{F}$ or $\Delta$.  Denote by $P_n$ and $C_n$
the path of length $n-1$ and the cycle of length $n$, respectively.
For distinct $x,y\in V(\Gamma)$, if $x$ is adjacent to $y$ in
$\Gamma$, then we write $x \sim_{\Gamma} y$ to denote this edge
$\{x,y\}$, or just by $x \sim y$ when $\Gamma$ is clear from the
context. In particular, if $P_n$ is the subgraph of $\Gamma$ induced
by the vertex subset $\{x_1,x_2,\ldots,x_n\}$ with the only
adjacencies being $x_i\sim x_{i+1}$, for each $1\le i\le n-1$, we denote
the induced path $P_n$ by
\[
x_1\sim x_2\sim\cdots \sim x_n.
\]

Similarly, in $\Gamma$, if the $n$-cycle
$C_n$ is the subgraph induced by the subset of vertices
$\{x_1,x_2,\ldots,x_n\}$, where the only edges between them are
$\{x_1,x_{n}\}\in E(\Gamma)$ and $\{x_i,x_{i+1}\}\in
E(\Gamma)$ for all $1\le i \le n-1$, we denote it by
\[
x_1\sim x_2\sim\cdots \sim x_n\sim x_1.
\]

The \emph{neighbourhood} of a vertex $v$ is $N(v)=\{w\in V(\Gamma):w\sim v\}$,
and the \emph{closed neighbourhood} of $v$ is $\bar N(v)=\{v\}\cup N(v)$.
A vertex $v$ of a graph $\Gamma$ is said to be \emph{dominating} if
$v\sim w$ for all $w\in V(\Gamma)\setminus\{v\}$.

\part{Recognising the commuting graph}

\section{The Recognition Problem}

The main focus of this part is the \emph{recognition problem} of
commuting graphs. It is basically an algorithmic question. Given an
undirected graph $X=(V,E)$ as input, the aim is to design an algorithm
that efficiently checks if there is a group $G$ with $|V|$ elements
such that the input graph $X$ is isomorphic to $\Gamma(G)$, and if so
to efficiently determine such an isomorphism.

We summarize our results.

\begin{itemize}
\item We obtain a deterministic polynomial-time algorithm for the case
  of extraspecial $p$-groups for odd primes $p$.  Extraspecial
  $p$-groups are a special case of $p$-groups of nilpotence class $2$.
  A formal definition is in Section~\ref{extrasp-sec}.
\item For the general case, we obtain a quasipolynomial-time algorithm
  for the recognition problem of commuting graphs. The algorithm is
  essentially based on a simple $2^{O(\log^3n)}$ time algorithm that
  enumerates the multiplication tables of all distinct groups of order
  $n$ (where by distinct we mean mutually non-isomorphic groups of
  order $n$). This enumeration algorithm combined with Babai's
  quasipolynomial time algorithm for graph isomorphism \cite{B16}
  yields a simple $2^{O(\log^3n)}$ time algorithm to check if a given
  $n$-vertex graph $X$ is the commuting graph of a group of order~$n$:
  for each $n$-element group $G$ that is output by the enumeration
  algorithm we check if $\Gamma(G)$ is isomorphic to $X$ by using
  Babai's isomorphism test \cite{B16}.
\item A natural question in connection with our algorithm for recognizing
the commuting graphs of extraspecial groups is whether groups with the
same commuting graphs are isoclinic. This holds for extraspecial
groups (this is exploited by the algorithm), and it is natural to
conjecture that this property holds for all groups of nilpotence class
2 (of which extraspecial groups are a subclass). We present
counter-examples for class-3 nilpotent groups and conjecture that the
property holds for class-2 nilpotent groups.
\item Additionally, we also show an efficient reduction of the
  commuting graph recognition problem to the recognition problem of
  commuting graphs of indecomposable groups. We also note recognizing
  the commuting graph of generalized dihedral groups along with an extension
  to Frobenius groups.
\end{itemize}

\noindent\textbf{Some related work.}~ There is a result by Giudici and
Kuzma \cite{GK16} that shows the following: every $n$-vertex graph $X$
with at least two vertices of degree $n-1$ is realizable as the
commuting graph of a semigroup. It is easy to see that their
construction actually gives a polynomial-time algorithm for finding a
semigroup with $n$ elements and a bijection from it to the vertex set
of $X$ such that the edges of $X$ realize the commuting relation of
the semigroup. It is a nearly complete answer to the question in the
semigroups setting.

 Solomon and Woldar \cite{SW13} have shown that the commuting graph
 $\Gamma(G)$ of a finite simple group $G$ uniquely determines
 the group. More precisely, they have shown that for any finite group
 $H$ and any finite simple group $G$, their commuting graphs
 $\Gamma(H)$ and $\Gamma(G)$ are isomorphic if and only if the groups
 $G$ and $H$ are isomorphic. An interesting question is whether or not
 the proof of this result \cite{SW13} yields a polynomial-time
 algorithm for checking if a given graph $X$ is the commuting graph of
 a simple group.

\section{Basic properties}\label{basic-sec}

We being with some preliminary observations that are well-known in the
literature (see, e.g., the survey \cite{GGsurvey}).

Let $G$ be a finite group. We first describe the cliques of
$\Gamma(G)$. If a vertex subset $S$ induces a clique in $\Gamma(G)$
then $S$ is a commuting subset of elements of $G$. Conversely, every
commuting subset of elements of $G$ forms a clique in
$\Gamma(G)$. Which cliques of $\Gamma(G)$ correspond to subgroups of
$G$? Although we cannot directly infer the group multiplication from
$\Gamma(G)$, we can observe the following.

\begin{lemma}\label{maxclique-lem}
  Let $G$ be a finite group and $\Gamma(G)$ its commuting graph. A
  vertex subset $S$ is a \emph{maximal} clique of $\Gamma(G)$ if and only if $S$
  is a maximal abelian subgroup of $G$.
\end{lemma}

\begin{proof}
  Suppose $H$ is a maximal abelian subgroup of $G$. Clearly, $H$ is a
  clique in $X$. If $x\notin H$ is adjacent to all of $H$ then $x$
  commutes with each element of $H$ implying that $\angle{H\cup\{x\}}$
  is an abelian subgroup of $G$ strictly larger than $H$. Hence the
  clique induced by $H$ is maximal.

  Conversely, suppose $S\subset G$ is a maximal clique in $\Gamma(G)$.
  Then any two elements $x,y\in S$ commute with each other. Hence, the
  subgroup $H=\angle{S}$ generated by $S$ is an abelian subgroup of
  $G$ which means that the vertex subset consisting of elements of $H$
  is a clique in $\Gamma(G)$ containing $S$. As $S$ is maximal, by
  assumption, it follows that $S=H$ which means $S$ is a maximal
  abelian subgroup. \qed
\end{proof}

\subsection{The vertex degrees of $\Gamma(G)$ and conjugacy classes of $G$}

Let $X=(G,E)$ be the commuting graph of a finite group $G$.  Let $x^G$
denote the \emph{conjugacy class} of $x\in G$:
\[
x^G = \{g^{-1}xg\mid g\in G\},
\]
which is the orbit of $x$ under the conjugation action of $G$ on
itself. Let $\deg(v)$ denote the degree of a vertex $v$ in the graph
$X$.
The orbit-stabilizer lemma \cite{permgps} directly implies the following.

\begin{prop}
  For each $x\in G$ its centralizer $C_G(x)$ is the closed neighborhood
  $\bar{N}(x)$ of $x$ in $\Gamma(G)$, and
  \[
|C_G(x)| = 1+ \deg(x) = {\frac{|G|}{|x^G|}}, \textrm{ for all }x\in G.
  \]
\end{prop}

Let $n=|G|$, let $m$ denote the number of edges in the commuting graph $\Gamma(G)$,
and let $k$ denote the number of conjugacy classes.  As $\sum_{x\in
  G}\deg(x) = 2m$, we have:

\begin{equation}
2m +n = \sum_{x\in G} {\frac{|G|}{|x^G|}} = |G|\cdot \sum_{x\in
  G}{\frac{1}{|x^G|}} = n\cdot k.
\label{e:edges}
\end{equation}

Thus the number of conjugacy classes of $G$ is $k=(2m+n)/n$, which, by
the above equation, can be inferred from the commuting graph. Thus,
for example, the only regular commuting graphs are the complete
graphs, which are the commuting graphs of abelian groups.

Let $f(n)$ denote the minimum number of conjugacy classes in a group
of order~$n$. The function $f(n)$ is quite well
studied. Landau~\cite{La} showed that $f(n)\to\infty$ as
$n\to\infty$. The first lower bound was obtained by
Brauer~\cite{Brauer} and Erd\H{o}s and Tur\'an~\cite{ET}, who showed
that $f(n)\ge\log\log n$ (logarithms to base~$2$). This was improved
to $\epsilon\log n/(\log\log n)^8$ by Pyber~\cite{Py}. The exponent
$8$ in the previous expression was reduced to $7$ by Keller
\cite{Keller}, and then to $3+\epsilon$ by Baumeister,
Mar\'oti and Tong-Viet~\cite{BMT-V}, for any $\epsilon>0$. It is
conjectured that a bound of the form $f(n)\ge C\log n$ holds for some
constant~$C$.

This is relevant to the recognition problem because of the following
observation, which follows immediately from equation \eqref{e:edges}.

\begin{prop}
  An $n$-vertex graph $X$ with fewer than $\frac{n(f(n)-1)}{2}$ edges
  cannot be the commuting graph of a group, where $f(n)$ is the
  minimum number of conjugacy classes that a group of order~$n$ can have.
\end{prop}

At the other extreme, we can rule out very dense incomplete graphs by
the $5/8$-theorem \cite{Gu73} for finite groups. More precisely, for
any finite non-abelian group $G$ with $n$ elements:
\[
|C|=|\{(x,y)\mid x,y\in G \textrm{ and } xy=yx\}|\le \frac{5}{8}\cdot n^2,
\]
where $C\subset G\times G$ denotes the set of ordered pairs $(x,y)$ of
commuting elements. Then, in the commuting graph $\Gamma(G)=(G,E)$
the number of edges
\[
|E|=\frac{|C|-n}{2}\le \frac{5}{16}\cdot n^2 -\frac{n}{2}\le \frac{5}{8}\cdot{n\choose 2}.
\]
Putting it together, we have:
\begin{prop}
  An incomplete graph $X=(V,E)$ cannot be the commuting graph of a group
  if $|E|>\frac{5}{8}\cdot {|V|\choose 2}$.
\end{prop}

\subsection{The commuting graph and maximal abelian subgroups}

For a finite group $G$, let $2^G$ denote its power set and let
$\mathcal{M}\subseteq 2^G$ denote the set of all \emph{maximal abelian
  subgroups} of $G$. Associated with $G$ is the natural
\emph{hypergraph} $(G,\mathcal{M})$, where the hyperedges are
precisely the maximal abelian subgroups of $G$.

\begin{prop}\label{max-abel-subgps}
  The commuting graph of a finite group determines the hypergraph of
  its maximal abelian subgroups, and, conversely the maximal abelian
  subgroups hypergraph of the group determines its commuting graph.
\end{prop}

\begin{proof}
  Let $G$ be a finite group. The maximal abelian subgroups are the maximal
  cliques in $\Gamma(G)$; and the edges of $\Gamma(G)$ are the $2$-element
  subsets of maximal abelian subgroups.\qed
\end{proof}

\begin{remark}
Since the commuting graph $\Gamma(G)$ of a finite simple group $G$
uniquely determines the group \cite{SW13}, by Proposition~\ref{max-abel-subgps}
it follows that the set of maximal abelian groups of a finite simple
group $G$ uniquely determines $G$.

For a finite group $G$, the number of maximal abelian subgroups is
bounded by ${|G|}\choose {\log|G|}$ because every subgroup of $G$ has
a generating set of size bounded by $\log|G|$. This simple bound is tight apart from a constant in the exponent. To
see this, consider the extraspecial group $G$ of order $p^{2n+1}$ and
exponent $p$, where $p$ is an odd prime. The centre has order $p$, and
$G/Z(G)$ is isomorphic to a $2n$-dimensional vector space $V$ over the
field $F$ of $p$ elements, with the bilinear form from $V\times V$ to
$F$ corresponding to the commutation map from $G/Z(G)\times G/Z(G)$ to
$Z(G)$. Maximal abelian subgroups contain the centre, and correspond
to maximal totally isotropic subspaces of $V$. It is known that the
number of such subspaces is $\prod_{i=1}^n(p^i+1)$ (see
\cite{Tay}). As
\[
  \prod_{i=1}^n(p^i+1) > p^{n(n+1)/2}>p^{n(2n+1)/4}=|G|^{n/4},
\]
it follows that the number of maximal abelian subgroups of the
extraspecial $p$-group $G$ is $|G|^{\Omega(\log |G|)}$ for any odd
prime $p$.
\end{remark}

\section{Commuting graphs of product groups}\label{prod-sec}

A natural approach to the problem of recognizing commuting graphs is
to examine the connection between product decompositions of the
commuting graph $\Gamma(G)$ and the direct product decomposition of
the group $G$.

Let $G$ and $H$ be finite groups. Consider the commuting graph
$\Gamma(G\times H)$ of their direct product. We recall the
definition of strong products of graphs.

\begin{defn}{\rm\cite[Chapter 4, Page 35]{IK08book}}
Let $X=(V,E)$ and $X'=(V',E')$ be simple undirected graphs. Their
\emph{strong product}, denoted $X\boxtimes X'$, is a simple undirected
graph with the cartesian product $V\times V'$ as its vertex set and
edges defined as follows: distinct pairs $(u,u')$ and $(v,v')$ are
adjacent if and only if one of the following holds:
\begin{itemize}
\item $u=v$ and $(u',v')\in E'$,
\item $u'=v'$ and $(u,v)\in E$,
\item $(u,v)\in E$ and $(u',v')\in E'$.
\end{itemize}
\end{defn}

The following proposition is immediate from the definition.

\begin{prop}\label{cgraph-strong}
  For finite groups $G$ and $H$
  \[
  \Gamma(G\times H) = \Gamma(G)\boxtimes \Gamma(H).
  \]
\end{prop}

The strong product is both \emph{commutative} and \emph{associative}
\cite[Chapter 4]{IK08book}, in the sense that $X\boxtimes Y\cong
Y\boxtimes X$ and $(X\boxtimes Y)\boxtimes Z\cong X\boxtimes
(Y\boxtimes Z)$.

By a \emph{non-trivial graph} we mean a graph with more than one
vertex. Given a graph $X$, a \emph{factorization} $X\cong Y\boxtimes
Z$ is \emph{non-trivial} if both $Y$ and $Z$ are non-trivial. A graph
$X$ is said to be a \emph{prime graph} if it does not have a
non-trivial factorization. (We warn the reader that this term is also used
for graphs defined by arithmetic parameters of a finite group; see for
example~\cite{Mas_pg}.)

It turns out that any connected graph $X$ has a unique factorization
(up to isomorphism) into a strong product of prime graphs \cite[Theorem
  7.14]{IK08book}. Moreover, this prime factorization can be
computed in polynomial time \cite{FeSch92}. As a consequence, we
obtain the following polynomial-time reduction.

Recall a group $G$ is said to be \emph{indecomposable} if it
cannot be expressed as the direct product of two non-trivial groups.

\begin{theorem}\label{product-thm}
  The problem of recognizing the commuting graphs of groups is
  polynomial-time reducible to the problem of recognizing the
  commuting graphs of indecomposable groups.
\end{theorem}

\begin{proof}
  Suppose we have an algorithm $\mathcal{A}$ for recognizing the
  commuting graphs of indecomposable groups. Using $\mathcal{A}$ as
  subroutine, we present a polynomial-time algorithm for recognizing
  the commuting graphs of all finite groups.

  Let $X=(V,E)$ be an undirected graph on $n$ vertices which is a
  purported commuting graph.

  First, using the polynomial-time algorithm of Feigenbaum and
  Sch\"affer \cite{FeSch92} we can factorize $X$ as
  \[
  X=X_1\boxtimes X_2\boxtimes\cdots \boxtimes X_k,
  \]
  where each $X_i$ is a prime graph on at least two vertices.  It
  follows that $k\le\log n$. Now, for each subset $S\subseteq [k]$
  of the prime graphs we combine them by taking the strong
  direct product to define the graph
  \[
  X_S =\boxtimes_{i\in S} X_i.
  \]
  Notice that any order in which the strong product of these graphs
  $X_i$ is computed yields the same graph, up to isomorphism.

  Thus, in time which is polynomial in $n$, we have computed graphs $X_S$ for
  each subset $S$ of $[k]$.  Now, we invoke the subroutine
  $\mathcal{A}$ that checks if $X_S$ is the commuting graph of an
  indecomposable group $G_S$, and if so, finds a labeling of the
  vertices of $X_S$ with elements of $G_S$ consistent with the
  commuting relations.

  We can now check if the input graph $X$ is the commuting graph of a
  group with a straightforward dynamic programming strategy based on
  the following claim which directly follows from
  Proposition~\ref{cgraph-strong}.

\begin{claim}
  For any two disjoint subsets $S,S'$ of $[k]$ such that $X_S$ and
  $X_{S'}$ are the commuting graphs of groups $G_S$ and $G_{S'}$ the
  graph $X_S\boxtimes X_{S'}$ is the commuting graph of the direct
  product group $G_S\times G_{S'}$.
\end{claim}

Now, the algorithm works in stages, computing subsets $S$ of $[k]$
along with the graph $X_S$ and group $G_S$ such that
$X_S=\Gamma(G_S)$.
\begin{enumerate}
\item \textbf{for} stages $0$ \textbf{to} $k$ \textbf{do}
\item \textbf{Stage $0$}~~ we have subsets $S$ such that $G_S$ is an
  indecomposable group. We mark all such subsets $S$ as true. We mark
  the remaining subsets as false.
\item \textbf{Stage $i+1$}~~ For each pair of disjoint subsets $S$ and
  $S'$ marked true in Stages $1,2,\ldots,i$, such that $S\cup S'$ is
  marked false, we mark $S\cup S'$ true and compute $X_{S\cup
    S'}=X_S\boxtimes X_{S'}$ and $G_{S\cup S'}=G_S\times G_{S'}$.
\item \textbf{end-for}
\item If $[k]$ is marked true then the input $X$ is the commuting
  graph of the group $G_{[k]}$ computed above.
\end{enumerate}

The above description checks if $X$ is the commuting graph of a group
with at most $2^k\le n$ calls to the subroutine $\mathcal{A}$ and the
running time of the remaining computation is clearly polynomially
bounded in $n$.\qed
\end{proof}

As mentioned in the introduction, the Solomon--Woldar theorem
\cite{SW13} shows that the commuting graph of a finite simple group
uniquely determines it. Let the group
$G=G_1\times G_2\times\cdots \times G_t$ be the direct product of
finite simple groups $G_i, 1\le i\le t$. By
Proposition~\ref{cgraph-strong} we can express its commuting graph
$\Gamma(G)$ as
\[
\Gamma(G)= \Gamma(G_1)\times \Gamma(G_2)\times\cdots \times
\Gamma(G_t)
\]
using the strong product. Putting it together, we can deduce the
following.


\begin{cor}
  If $G$ is a direct product of simple groups, then $G$ is uniquely
  determined by its commuting graph $\Gamma(G)$.
\end{cor}

We do not know if the commuting graph of an indecomposable group $G$
is necessarily a prime graph. However, if $G$ is simple then it turns
out, as we show below, that $\Gamma(G)$ is a prime graph.

\begin{lemma}\label{simple-prime-lem}
The commuting graph of a finite simple group is a prime graph under the
strong product.
\end{lemma}

\begin{proof}
We use the fact that, if $G$ is a non-abelian simple group and
$g\in G\setminus\{1\}$, then there exists $h\in G$ such that
$\langle g,h\rangle=G$~\cite{GK}. Now if $\langle g,h\rangle=G$, then
\begin{itemize}
\item $g$ and $h$ are nonadjacent in the commuting graph (since $G$
is non-abelian);
\item $g$ and $h$ have no non-identity common neighbour in the commuting
graph (since a common neighbour would belong to $Z(G)$, but $Z(G)=\{1\}$.
\end{itemize}

Now suppose for a contradiction that $\Gamma(G)$ is the strong product
of two non-trivial graphs with vertex sets $A$ and $B$ (i.e., with
$|A|>1$ and $|B|>1$). Then we can identify $G$, as the vertex set of
$\Gamma(G)$, with the Cartesian product $A\times B$. Suppose that
$(a,b)$ is the identity of $G$. Then $(a,b)$ is joined to all other
vertices in the commuting graph.

It follows that $a$ is joined to every other vertex in $A$, and $b$ to
every other vertex in $B$; therefore, for all $x\in A\setminus\{a\}$,
$y\in B\setminus\{b\}$, the three vertices $(a,y)$, $(x,b)$ and
$(x,y)$ are adjacent to each other in the commuting graph $\Gamma(G)$.

Choose $y\in B\setminus\{b\}$, and suppose that
$\langle(a,y),(u,v)\rangle=G$. Now
\begin{itemize}
\item if $u=a$ then $(a,y)$ and $(a,v)$ are both joined to $(x,b)$
for any $x\in A\setminus\{a\}$;
\item if $v=b$ then $(a,y)$ and $(u,b)$ are joined;
\item if neither of the above, then $(a,y)$ and $(u,v)$ are both
joined to $(u,b)$.
\end{itemize}
Each case is contradictory; so our assumption that $\Gamma(G)$ is the
strong product of two non-trivial graphs is false, and the theorem is
proved.\qed
\end{proof}

\begin{remark}
  The proof of Lemma~\ref{simple-prime-lem} only requires that
  $Z(G)=1$ and that any non-identity element is contained in a
  $2$-element generating set. These assumptions are valid in any
  almost simple group $G$ with simple normal subgroup $S$ such that
  $G/S$ is cyclic~\cite{BGH}.
\end{remark}

As a consequence of Lemma~\ref{simple-prime-lem}, we obtain the
following variant of Theorem~\ref{product-thm}.

\begin{theorem}
  The problem of recognizing the commuting graph of a direct product
  of simple groups is polynomial-time reducible to the problem of
  recognizing the commuting graph of a simple group.
\end{theorem}

The algorithm for the reduction in the above theorem is much simpler
and more efficient than in Theorem~\ref{product-thm}. For we only need to
consider the prime factors $X_i$, in the factorization
$X=\boxtimes_{i=1}^k X_i$, and not all the strong products $X_S$ for
arbitrary subsets $S\subseteq [k]$.

\section{Commuting graphs of semidirect products}\label{semi-sec}

In this section we explore whether we can recognize the commuting
graphs of semidirect products $G=H\rtimes K$ if $H$ and $K$ are both
from group classes whose commuting graphs are easily recognizable.

For example, consider the commuting graph of the dihedral group
$D_{2n}$.  Let $D_{2n}=\angle{a,b}$ where $a^n=b^2=1$ and
$bab=a^{-1}$. For $n$ odd, $1$ is the only dominating vertex, there is
an $n$-clique corresponding to $\angle{a}$ and the $ba^i$ correspond
to pendant vertices. Thus, the commuting graph of $D_{2n}$ is easily
recognizable. For $n$ even, it is a bit different with $a^{n/2}$
being the other dominating vertex.

More generally, consider the generalized dihedral group $\angle{A,t}$
defined earlier, where $A$ is abelian of order $2n$, $n$ odd. Its
commuting graph is isomorphic to that of the dihedral group $D_{2n}$
of the same order. Hence, as observed above, we can recognize that a
graph is the commuting graph of a generalized dihedral group, but
cannot decide which one.

Seeking a generalization of this example we consider Frobenius groups.

\subsection{Commuting graphs of Frobenius groups}\label{frob-sec}

In this subsection we demonstrate that we can recognise from its
commuting graph that a group $G$ is a Frobenius group. These groups
were defined earlier. Let $G$ be a Frobenius group with kernel $N$
and complement $H$.

Thompson~\cite{Thompson} proved that the Frobenius kernel is
nilpotent, and Zassenhaus~\cite{Zassenhaus} worked out the detailed
structure of a Frobenius complement.  Everything we need about
Frobenius groups is explained in standard group theory texts
~\cite{passman,isaacs,serre}.

In the commuting graph of a group $G$, the identity is a dominating
vertex (that is, joined to all others); indeed, any vertex in the centre
is dominating, so if $Z(G)$ is non-trivial then the commuting graph is
$2$-connected.

\begin{lemma}\label{lem:frob}
\begin{enumerate}
\item A Frobenius kernel has non-trivial centre.
\item A Frobenius complement has non-trivial centre.
\end{enumerate}
\end{lemma}

\begin{proof}
  (a) By Thompson's theorem \cite[Theorem 1]{Thompson}, a Frobenius
  kernel is nilpotent, and nilpotent groups have a non-trivial centre.

\medskip

(b) Let $H$ be a Frobenius complement.
By \cite[Theorem 6.3, Corollary 6.10]{isaacs}\label{isaacs1},
if $H$ has even order, then it has a unique element of order~$2$, which
lies in the centre; while, if $H$ has odd order, then it has a unique
subgroup of order $p$ for every $p$ dividing $|H|$. In the latter case,
choose $p$ to be the smallest prime divisor of $H$, and $P=\angle{a}$ the
unique subgroup of order $p$. Then $P$ is a normal subgroup of $H$. For
any $h\in H$, we have $a^h=a^m$ for some $m$ with $1\le m\le p-1$. If
$m>1$, let $q$ be a prime divisor of $m$; then some power of $h$ has order
$q$, a contradiction. So $a^h=a$ for all $h\in H$, whence $a\in Z(H)$.\qed
\end{proof}

\begin{theorem}\label{frob-thm}
  Let $G$ be a group of order $nk$, where $n>1$, $k>1$, and $n$ and $k$ are
  relatively prime. Then $G$ is a Frobenius group with Frobenius
  complement of order $k$ if and only if its commuting graph
  $\Gamma(G)$ satisfies the following conditions:
\begin{itemize}
\item[(a)] there is a unique dominating vertex $1$ in $\Gamma(G)$,
  where $1$ corresponds to the identity element of $G$.
\item[(b)] $\Gamma(G)\setminus\{1\}$ has a component of size $n-1$ and
  $n$ components of size $k-1$, and each component has a dominating
  vertex.
\end{itemize}
\end{theorem}

\begin{proof}
  Let $G$ be a group of order $nk, n>1, k>1$ with $n$ and $k$
  relatively prime.

  We prove the forward direction first. Suppose that $G$ is a
  Frobenius group with Frobenius complement $H$ of order $k$ and
  Frobenius kernel $N$. Then, as $G=N\rtimes H$, $|G|=|N|\cdot |H|$
  which implies the kernel $N$ is of order $n$.  Consider the
  commuting graph $\Gamma(G)$. As the identity element commutes with
  every other element in $G$, it is a dominating vertex in
  $\Gamma(G)$.  We recall \cite[35.24~and~35.25]{Asch86} that
  $N \trianglelefteq G$, $C_G(h)\le H$ for each $h \in H$, and
  $C_G(x)\le N$ for each $x \in N\setminus \{1\}$. Therefore,
  non-identity vertices in $N$ do not commute with non-identity
  vertices in a complement. Thus the identity element is the unique
  dominating vertex in $\Gamma(G)$.  Moreover, non-identity vertices
  in different complements do not commute with each other. Thus the
  components of $\Gamma(G)\setminus\{1\}$ are as stated, and
  Lemma~\ref{lem:frob} shows that they have dominating vertices.

Conversely, suppose that the commuting graph $\Gamma(G)$ has
properties (a) and (b). If $C$ is a component with a dominating vertex
$c$, then the centralizer $C_G(c)$ is equal to $C\cup\{1\}$, which is
thus a subgroup of $G$.  Let $N$ be the subgroup containing the
component of size $n-1$, and let $H_1,\ldots,H_n$ be the subgroups
containing the other components. Note that $|N|=n$ and $|H_i|=k, 1\le
i\le n$.

Since $\Gamma(G)$ is invariant under automorphisms of $G$, we have an
action of $G$ on the set $\Omega$ of components of
$\Gamma(G)\setminus\{1\}$ of size $k-1$ by conjugation. We show that
this action satisfies the conditions for a Frobenius group given
above.

Choose a prime $p$ dividing $k$. Since $n$ and $k$ are relatively
prime, $p$ does not divide $n$. So the subgroup $N$ cannot contain a
Sylow $p$-subgroup; but each of $H_1,\ldots,H_n$ contains such a
subgroup. By the conjugacy part of Sylow's theorem, $G$ acts
transitively on $\Omega$.

For $1\le i\le n$, each element of $H_i\setminus\{1\}$ acting by
conjugation, fixes itself and so fixes the component induced by
$H_i\setminus\{1\}$. That is, a non-identity element of
$G\setminus N$, acting by conjugation, fixes itself, and so fixes the
component containing it. We must show that it fixes no other
component. To that end, we count the fixed points of elements of $G$
acting on $\Omega$. The identity fixes $n$; each non-identity element
of $N$ fixes some non-negative number; and each of the remaining
elements $G\setminus N$ fix a positive number of elements. So the sum
of the fixed point numbers is at least $n+n(k-1)=nk$. But, by the
Orbit-Counting Lemma, this sum is equal to $|G|=nk$, since $G$ is
transitive. Consequently, non-identity elements of $N$ do not fix any
point of $\Omega$, and the remaining non-identity elements fix exactly
one point each. So $G$ is a Frobenius group.\qed
\end{proof}

Suppose $X$ is a graph given as input such that $X=\Gamma(G)$ for some
group $G$. Then Theorem~\ref{frob-thm} gives us a polynomial-time test
for checking if the group $G$ is a Frobenius group.

\begin{remark}
Above we assumed that our graph is a commuting graph. However, it is easy
to construct a graph satisfying the conditions of Theorem~\ref{frob-thm}
which is not a commuting graph. Can we find necessary and sufficient
conditions for $\Gamma$ to be the commuting graph of a Frobenius group?
\end{remark}

\section{Recognizing commuting graphs of extraspecial groups}\label{extrasp-sec}

In this section we design a polynomial-time algorithm that takes a
simple undirected graph $X=(V,E)$, with $n=p^{2r+1}$ vertices,
$r\ge 2$ and $p$ an odd prime, as input and determines if $X$ is the
commuting graph of an extraspecial group $G$ of order $p^{2r+1}$ (see
Section~\ref{prelim-sec} for definition). The algorithm will also compute a
certifying bijection $v\mapsto g_v$ labeling each vertex $v\in V$ by a
unique group element $g_v\in G$ such that the commuting relations
hold.

There are exactly two non-isomorphic extraspecial groups of order
$p^{2r+1}$, for each $r\ge 1$ and each prime $p$. These two groups,
which we denote as $G_1$ and $G_2$, are described below.

\begin{itemize}
\item The first group $G_1$ is generated by $2r+1$ elements
  $G_1=\angle{z,x_i,y_i, 1\le i\le r}$ such that $[x_i,x_j]=1$,
  $[y_i,y_j]=1$ for all $i, j$, and $[x_i,y_j]=1$ for all $i\ne
  j$, $[x_i,y_i]=z$ for all $i$, and $x_i^p=y_i^p=z^p=1$ for all
  $i$. The centre $Z(G_1)=\angle{z}$.

\item The second group $G_2$ also has $2r+1$ generators
  $G_2=\angle{z,x_i,y_i, 1\le i\le r}$ such that $[x_i,x_j]=1$,
  $[y_i,y_j]=1$ for all $i, j$, and $[x_i,y_j]=1$ for all $i\ne
  j$. And $[x_i,y_i]=z$ for all $i$, and $x_i^p=1=z^p$, $y_i^p=z$ for
  all $i$. The centre $Z(G_2)=\angle{z}$.
\end{itemize}

Each element $g\in G_1\setminus \{1\}$ has order $p$. On the other
hand, each element $g\in G_2\setminus \{1\}$ has order either $p$ or
$p^2$. Hence, $G_1$ and $G_2$ are non-isomorphic.
We refer to Gorenstein~\cite[Chapter 5]{Gorenstein} for more about
extraspecial groups.

We now analyze the commuting graph $\Gamma(G)$ for $G\in\{G_1,G_2\}$
of order $n=p^{2r+1}$. The vertices of $\Gamma(G)$ corresponding to
the centre $Z(G)=Z=\angle{z}$ are precisely the $p$ vertices of degree
$p^{2r+1}-1$ each. In either case ($G\in\{G_1,G_2\}$) we have:
\begin{itemize}
  \item The quotient group $G/Z$ is an elementary abelian $p$-group of order
    $p^{2r}$.
  \item The centre $Z$ can be identified with the finite field
$\F_p$ and $G/Z$ can be identified with the $2r$-dimensional vector
space $V(2r,p)=\F_p^{2r}$.
\end{itemize}

Furthermore, as $Z=G'$, for $x,y\in G$ we can write their commutator
$[x,y]=z^{i_{x,y}}$ for some $0\le i_{x,y}\le p-1$. Hence, the
commutation map
\[
(xZ(G),yZ(G))\mapsto [x,y]
\]
from $G/Z\times G/Z \to Z$ can be identified with a \emph{bilinear
form} $\sympl{,}: V(2r,p)\times V(2r,p) \to \F_p$ that satisfies
the following:

\begin{itemize}
\item It is \emph{symplectic}: We have $\sympl{x,y}=-\sympl{y,x}$ for
  all $x,y\in V(2r,p)$. In particular, $\sympl{x,x}=0$ for all $x\in
  V(2r,p)$.
\item It is \emph{non-degenerate}: there is no non-zero $u\in V(2r,p)$
  such that $\sympl{u,v}=0$ for all $v\in V(2r,p)$.
\end{itemize}

\begin{defn}\label{sympl-space-def}
  The vector space $V(2r,p)$ equipped with a symplectic form
  $\sympl{,}$ is a \emph{symplectic space}.
\end{defn}

If $\sympl{u,v}=0$ we say $u$ is \emph{orthogonal} to $v$. For a
subset $S$ of the symplectic space $V(2r,p)$ let
\[
S^\perp = \{u\in V(2r,p)\mid \sympl{u,v}=0 \text{ for all } v\in S\}
\]
denote the subspace of $V(2r,p)$ consisting of vectors
\emph{orthogonal} to each vector in $S$.

\begin{defn}\label{sympl-basis-def}
  A \emph{symplectic basis} for the symplectic space $V(2r,p)$ is a
  basis of $2r$ vectors $\bar{x}_i, \bar{y}_i, 1\le i\le r$ such
  that:
\begin{itemize}
\item $\sympl{\bar{x}_i,\bar{y}_i}=1$ for all $i$.
\item $\sympl{\bar{x}_i,\bar{y}_j}=0$ for all $i\ne j$.
\item $\sympl{\bar{x}_i,\bar{x}_j}=0$ for all $i,j$.
\item  $\sympl{\bar{y}_i,\bar{y}_j}=0$ for all $i,j$.
\end{itemize}
\end{defn}

The symplectic space $V(2r,p)$ has a symplectic basis \cite[Section
3.4.4]{Wilson-book}.

\begin{remark}
  We will describe an elementary self-contained polynomial-time
  algorithm for recognising commuting graphs of extraspecial
  groups. We have an alternative algorithm, based on the
  Buekenhout-Shult theorem \cite{bs} classifying certain point-line
  geometries, that exploits the finite geometry structure of the
  quotient group $G/Z(G)$. However, we present the former in the
  interest of keeping the proof entirely self-contained in the
  context of this paper.
\end{remark}

\begin{lemma}\label{G-to-sympl}
  Given as input the multiplication table of an extraspecial group $G$
  of order $p^{2r+1}$, in time polynomial in $p^{2r+1}$ we can
  construct its associated symplectic form $\sympl{,}$, as an explicit
  bilinear map from $\F_p^{2r}\times \F_p^{2r}\to \F_p$.
\end{lemma}

\begin{proof}
  Let $n=p^{2r+1}$. First, we compute the centre $Z=\angle{z}$ of $G$
  from the multiplication table. In order to compute $G/Z$ we can find
  its multiplication table from the given multiplication table of $G$.
  We can find a decomposition of $G/Z$ as the direct product of $2r$
  cyclic groups of order $p$ by finding \emph{independent generators}
  for $G/Z$ using a simple algorithm. Suppose
  $g_1,g_2,\ldots,g_j\in G/Z$ are independent (i.e. the identity can
  be expressed as a product of their powers only if the exponents are
  all zero). If $j<2r$ then the subgroup
  $H_j=\angle{g_1,g_2,\ldots,g_j}$ which is of size $p^j$ can be
  listed in time polynomial in $n$. We can pick $g_{j+1}$ as any
  element in $G\setminus H_j$ and continue the process until we pick
  $2r$ such independent generators which will generate $G/Z$.  That
  gives us an efficient identification of $G/Z$ with $V(2r,p)$.  For
  $xZ,yZ\in G/Z$ we will have corresponding vectors
  $v_x,v_y\in V(2r,p)$. If the commutator $[x,y]=z^{i_{x,y}}$ then we
  can define the symplectic form as $\sympl{v_x,v_y}=i_{x,y}$, where
  $i_{x,y}\in \F_p$. This completes the proof. \qed
\end{proof}

\begin{defn}\label{orth-graph-def}
The \emph{symplectic graph} of the symplectic space $V(2r,p)$
(equipped with symplectic form $\sympl{,}$) is an undirected simple
graph with $V(2r,p)$ as vertex set such that for each pair of distinct
vectors $u,v\in V(2r,p)$, $(u,v)$ is an undirected edge in the graph
precisely when $\sympl{u,v}=0$.
\end{defn}

In particular, let $S_1$ and $S_2$ denote the \emph{associated
  symplectic graphs} of the two extraspecial groups $G_1$
and $G_2$ of order $p^{2r+1}$, respectively.

\begin{lemma}\label{lem:iso}\hfill{~}
  \begin{enumerate}
  \item The commuting graphs $\Gamma(G_1)$ and $\Gamma(G_2)$
    are isomorphic.
  \item The associated symplectic graphs $S_1$ and $S_2$ are isomorphic.
\end{enumerate}
\end{lemma}

\begin{proof}
  Recall, as defined at the beginning of this section, that both
  $G_1=\angle{x_1,x_2,\ldots,x_r,y_1,y_2,\ldots,y_r,z}$ and
  $G_2=\angle{x_1,x_2,\ldots,x_r,y_1,y_2,\ldots,y_r,z}$. The commuting
  relations between the generators are identical. The only difference
  is that in $G_2$ each $y_i$ is of order $p^2$, and in $G_1$ they are
  all order $p$.

  Because of the commutation relation $[x_i,y_i]=z$ in both groups,
  where $Z=\angle{z}$ is the centre, each element in both groups
  can be expressed as products of the form
  \[
  x_1^{e_1}x_2^{e_2}\cdots x_r^{e_r}y_1^{f_1}y_2^{f_2}\cdots y_r^{f_r}\cdot z^c,
  \]
  where $0\le c,e_i,f_i\le p-1$ for all $i$. Suppose now we have two
  elements $x=\prod_{i=1}^r x_i^{e_i}\cdot \prod_{i=1}^r
  y_i^{f_i}\cdot z^c$ and $y=\prod_{i=1}^r x_i^{e'_i}\cdot
  \prod_{i=1}^r y_i^{f'_i}\cdot z^{c'}$.  Then, using the commutation
  relation $[x_i,y_i]=z$, which holds in both $G_1$ and $G_2$, we can
  check that $x$ and $y$ commute if and only if
  $\sum_{i=1}^re_if'_i=\sum_{i=1}^r f_ie'_i~(\Mod p)$. It follows that
  $(x,y)$ is an edge in $\Gamma(G_1)$ if and only if it is an edge in
  $\Gamma(G_2)$.  Hence the commuting graphs $\Gamma_1$ and $\Gamma_2$ are
isomorphic.

  Let $V_1$ and $V_2$ denote the symplectic spaces corresponding to
  the groups $G_1$ and $G_2$, respectively. For $G\in\{G_1,G_2\}$, as
  already observed, $G/Z$ is isomorphic to $V(2r,p)$ (as an additive
  group). Let $x_iZ\mapsto \bar{x}_i$ and $y_iZ\mapsto \bar{y}_i$,
  $1\le i\le r$, under this isomorphism, for $\bar{x}_i, \bar{y}_i\in
  V(2r,p)$ and $1\le i\le r$. Then, $xZ$ for
\[
  x=\prod_{i=1}^r
  x_i^{e_i}\cdot \prod_{i=1}^r y_i^{f_i}\cdot z^c
\]
maps to $\bar{x}=\sum_{i=1}^r(e_i\bar{x}_i+f_i\bar{y}_i)$. Likewise,
$y=\prod_{i=1}^r x_i^{e'_i}\cdot \prod_{i=1}^r y_i^{f'_i}\cdot z^{c'}$
maps to
$\bar{y}=\sum_{i=1}^r(e'_i\bar{x}_i+f'_i\bar{y}_i)$. Therefore, we can
evaluate $\sympl{\bar{x},\bar{y}}$ using the fact that
$\sympl{\bar{x}_i,\bar{x_j}}=0$, $\sympl{\bar{y}_i,\bar{y_j}}=0$, for
all $i\ne j$ and for all $i$ we have $\sympl{\bar{x}_i,\bar{y}_i}=1$
and $\sympl{\bar{y}_i,\bar{x_i}}=-1$.
  \[
  \sympl{\bar{x},\bar{y}}=\sum_{i=1}^r (e_if'_i - e'_if_i)
\]
It follows that the vertices $\bar{x}$ and $\bar{y}$ are adjacent in
the symplectic graph $S_1$ if and only if they are adjacent in $S_2$. \qed
\end{proof}

We now show that the symplectic graph of an extraspecial group $G$ can
be computed in polynomial time, given just its commuting graph
$\Gamma(G)$ as input.

\begin{lemma}\label{commute-to-sympl}
  Given as input the commuting graph $\Gamma(G)$, of an extraspecial
  group $G$ of order $p^{2r+1}$, $p$ an odd prime, we can compute in
  polynomial time an undirected graph $X_s=(V_s,E_s)$ with
  $|V_s|=p^{2r}$ such that $X_s$ is the symplectic graph
  associated with $G$.
\end{lemma}

\begin{proof}
  In the commuting graph $\Gamma(G)$, which has vertex set $G$, there
  are exactly $p$ dominant vertices, corresponding to the $p$ elements
  of the centre $Z$. Let $Z$ also denote this subset of vertices of
  $\Gamma(G)$.

  Recall that the \emph{closed neighborhood} of a vertex $v\in G$ is defined
  as
\[
\bar{N}(v) = \{u\in V\mid (u,v)\in E\} \cup \{v\}.
\]
Vertices $v_1,v_2,\in G$ are \emph{closed twins} if
$\bar{N}(v_1)=\bar{N}(v_2)$. This equality clearly defines an
equivalence relation (the \emph{closed twins} equivalence) on
$G\setminus Z$. For each $h\in G\setminus Z$ let $[h]$ denote the
closed twins equivalence class containing $h$.

\begin{claim}
  For every $h\in G\setminus Z$ the closed twins equivalence class
  $[h]$ is equal to $H\setminus Z$, where $H$ is the subgroup of $G$
  generated by $Z\cup \{h\}$.
\end{claim}

To see this claim, let $h\in G\setminus Z$. Consider the subgroup $H$
of $G$ generated by $Z\cup \{h\}$, where the centre
$Z=\angle{z}$. Then $|H|=p^2$ and $H=\{z^ih^j\mid 0\le i, j \le p-1\}$.
Furthermore,
\[
  H\setminus Z=\{z^ih^j\mid 0\le i\le p-1, 1\le j\le p-1\}.
\]
Now, as $\gcd(j,p)=1$ for nonzero $j\le p-1$, it follows that any
$x\in G\setminus Z$ commutes with $z^ih^j\in H\setminus Z$ if and only
if $x$ commutes with $h$. Hence, every element of $H\setminus Z$ is a
closed twin of $h$ which implies that $H\setminus Z\subseteq [h]$.

Let $g\notin G\setminus H$. We will show that $g\notin [h]$ by
exhibiting an element of $G\setminus Z$ that commutes with $h$ but not
with $g$. That will show that $H\setminus Z=[h]$, which will complete
the proof of the claim.

As $g\notin H$, we have $Zg\notin H/Z$. Now, identifying $G/Z$ with
the symplectic space, the coset $Zh$ maps to a vector $\bar{x}_1$ of
$V(2r,p)$ and $H/Z$ maps to the $1$-dimensional subspace of $V(2r,p)$
spanned by $\bar{x}_1$. We can extend this vector to a symplectic
basis $\bar{x}_i, \bar{y}_i, 1\le i\le r$ of $V(2r,p)$ with the
relations $1=\sympl{x_i,y_i}=-\sympl{y_i,x_i}$ along with the other
relations (as in the proof of Lemma~\ref{lem:iso}).

As $Zg\not\in H/Z$, its image $u\in V(2r,p)$ in the symplectic space
is of the form
\[
  u = \sum_{i=1}^r (\alpha_i\bar{x}_i + \beta_i \bar{y}_i),
\]
where either $\alpha_i\ne 0$ for some $i\ne 1$ or some $\beta_j\ne 0$
with $1\le j\le r$.  In the former case, $\sympl{u,\bar{x}_i}\ne 0$ but
$\sympl{\bar{x}_1,\bar{x}_i}=0$, and in the latter case
$\sympl{\bar{x}_1,\bar{y}_j}=0$ but $\sympl{u,\bar{y}_j}\ne 0$.

It follows that there is an element $x\in G\setminus Z$ that commutes
with $h$ but not with $g$. Hence, $g\notin [h]$.

Thus, the equivalence classes defined by the closed-twins equivalence
identifies all $p^2-p$ size vertex subsets of $G\setminus Z$ that
corresponds to $H\setminus Z$ for each subgroup $H$ of order $p^2$
such that $Z\le H\le G$. The number of such equivalence classes is
\[
\frac{p^{2r+1}-p}{p^2-p} = 1 + p + p^2 +\cdots + p^{2r-1}.
\]

The preceding discussion immediately implies the following.

\begin{claim}\label{cl:bip} For any two order $p^2$ subgroups $H$
  and $H'$ such that $Z\le H\le G$ and $Z\le H'\le G$, the subgraph of
  $\Gamma(G)$ induced by $(H\setminus Z)\cup (H'\setminus Z)$ either
  forms a complete bipartite graph (with $H\setminus Z$ and
  $H'\setminus Z$ as the two parts) or there are no edges between
  these two parts.
\end{claim}

To obtain the corresponding symplectic graph $X_s=(V_s,E_s)$ we
apply the following three steps to $\Gamma(G)$.

\begin{itemize}
\item[(a)] Collapse each equivalence class in $G\setminus Z$ into a
  single vertex. This gives rise to a set $V_c$ of
  $\sum_{i=0}^{2r-1}p^i$ many vertices. The edges between vertices
  in $V_c$ are naturally inherited from $\Gamma(G)$ by Claim~\ref{cl:bip}.
\item[(b)] Include a new vertex $v_0$ corresponding to the $0$ element
  of $V(2n,p)$; this corresponds to $p$ dominant vertices of $X$.
\item[(c)] The process of collapsing the closed-twin equivalence
  classes identifies vectors $v\in V(2r,p)$ with all the $p-1$ non-zero
  scalar multiples $\alpha v, \alpha\in\F^*_p$. We restore the $p-1$
  copies by replacing each $v\in V_c$ by $p-1$ copies. The edges
  between these vertices are naturally inherited.
\end{itemize}

By construction, $X_s$ is the symplectic graph of the symplectic
form on $V(2r,p)$ defined by the commuting graph $\Gamma(G)$. \qed
\end{proof}

\subsection{Recognizing the symplectic graph}

We now show that the symplectic graph of a non-degenerate symplectic
form on $V(2r,p)$ can be recognized in polynomial time.

\begin{theorem}\label{ortho-recog-thm}
  Given a simple undirected graph $X=(V,E)$ on a vertex set $V$ of
  size $p^{2r}$, for an odd prime $p$, in time polynomial in the size
  of $X$ we can recognize if $X$ is the associated symplectic graph of
  an extraspecial group of order $p^{2r+1}$. Furthermore, the
  algorithm will also compute a bijection from $V$ to $V(2r,p)$ and
  determine an associated symplectic form $\sympl{,}$.
\end{theorem}

\begin{proof}
  The somewhat detailed proof is structured as a series of claims
  followed by a description of the algorithm and its analysis.

  Let $X=(V,E)$ be a graph with $p^{2r}$ vertices. Our algorithm will
  be recursive, with correctness proof based on an inductive
  argument. We know that the symplectic space $V(2r,p)$ has a
  symplectic basis of $2r$ vectors of the form
  $\bar{x}_i, \bar{y}_i, 1\le i\le r$ as defined in
  Definition~\ref{sympl-basis-def}.


We first note that given a symplectic form $\sympl{,}$ on $V(2r,p)$,
we can construct a symplectic basis for $V(2r,p)$ with the following
greedy procedure.
\begin{itemize}
\item Pick $\bar{x}_1\ne 0$ in $V(2r,p)$ arbitrarily.
\item Since $\sympl{,}$ is non-degenerate, we can find a vector
  $\bar{u}\in V(2r,p)$ such that $\sympl{\bar{x}_1,\bar{u}}=\alpha\ne
  0$.  Define $\bar{y}_1 =\alpha^{-1}\bar{u}$ such that
  $\sympl{\bar{x}_1,\bar{y}_1}=1$.
\item Then notice that the subspace $\{\bar{x}_1,\bar{y}_1\}^\perp$ is
  a $2r-2$ dimensional symplectic space $V'$ (w.r.t. the same
  symplectic form). We can continue with the basis construction,
  recursively applying the above two steps to $V'$.
\end{itemize}

Now, suppose $X=(V,E)$ is the graph given as input on $p^{2r}$
vertices. If $X$ is indeed the symplectic graph of $V(2r,p)$ for some
symplectic form $\sympl{,}$, we will describe a polynomial-time
recursive procedure that recovers a bijection from $V$ to $V(2r,p)$
along with a symplectic form $\sympl{,}$ as witness for $X$ being the
associated symplectic graph. We will then argue that if $X$ is not a
symplectic graph then this procedure detects that.

Essentially, the problem is to efficiently simulate the above two-step
basis construction procedure given only the graph $X=(V,E)$ as input.
As the symplectic space is non-degenerate, the zero vector $0$ is the
only vertex in $X$ adjacent to all others and is easily identified.
Let $x_1$ be any other vertex in $V$ and we choose a non-zero vector
$\bar{x}_1\in V(2r,p)$ as its image. Next, we will choose $y_1$ as any
vertex not adjacent to $x_1$ in the graph $X$ and choose a non-zero
vector $\bar{y}_1\in V(2r,p)$ as its image. Let

\[
V'=\{v\in V\mid (v,x_1)\in E \text{ and } (v,y_1)\in E\}.
\]
That is, $V'$ is the common neighborhood of $x_1$ and $y_1$.  We can
easily identify $V'$ in $X$. Let $X'$ be the subgraph of $X$ induced
by the vertex subset $V'$. We have the following simple observation.
\begin{claim}\label{symp-cl1}
 If $X$ is the symplectic graph of a $2r$-dimensional symplectic space
 over $\F_p$ then the graph $X'$ is the symplectic graph of a
 symplectic space of dimension $2r-2$ over $\F_p$.
\end{claim}

The algorithm will now recursively check that $X'$ is indeed the
symplectic graph of a $2r-2$-dimensional symplectic space $V(2r-2,p)$
over $\F_p$. If so, it will compute a labeling of the vertices of $V'$
by linear combinations $\sum_{i=2}^r(\alpha_i \bar{x}_i + \beta_i
\bar{y}_i), \alpha_i,\beta_i\in \F_p$ that is consistent with the
orthogonality relation of a symplectic form $\sympl{,}$. Here the
basis vectors computed by the algorithm are $\bar{x}_i,\bar{y}_i$ for
$2\le i\le r$. At the end of the recursive call, the algorithm computes a
bijection from $V'$ to $V(2r-2,p)$ that labels each vertex in $V'$ by
a distinct linear combination of the basis vectors.

Now, the remaining task for the algorithm is to find a labeling of the
vertices in $V\setminus V'$ consistent with a symplectic form on
$V(2r,p)$, assuming $X$ is a symplectic graph.

For each vertex $v$ we will denote its vector label by $\bar{v}$.  For
a labeled subset $S$ of vertices of $V$ we will use $\bar{S}\subseteq
V(2r,p)$ to denote the set of vectors labeling vertices in $S$.

\begin{claim}\label{symp-cl2}
  A vertex $v\in V\setminus V'$ can be labeled by a non-zero vector
  $\alpha \bar{x}_1 + \beta \bar{y}_1$ in $V(2r,p)$,
  $\alpha,\beta\in\F_p$, if and only if $(v,u)\in E$ for all
  $u\in V'$.
\end{claim}

\claimproof{Consider the symplectic graph of $V(2r,p)$. Let
  $V(2r-2,p)$ denote the subspace spanned by $\bar{x}_i,\bar{y}_i, 2\le i\le n$.
  Clearly, every vector of the form $\alpha \bar{x}_1 + \beta \bar{y}_1,
  \alpha,\beta\in\F_p$ is orthogonal to each vector in
  $V(2r-2,p)$. Conversely, consider a vector $\alpha \bar{x}_1 + \beta \bar{y}_1
  +v\in V(2r,p)$, where $v\in V(2r-2,p)$ is non-zero. Since $V(2r-2,p)$
  is non-degenerate, there is a $u\in V(2r-2,p)$ such that
  $\sympl{v,u}\ne 0$ which implies $\sympl{\alpha \bar{x}_1 +\beta \bar{y}_1 +
    v,u}=\sympl{v,u}\ne 0$.}

Thus, the algorithm can find a subset $V_{12}$ of precisely $p^2-1$
many such vertices in $V$ that are adjacent to all of $V'$, of which
we have already labeled two vertices as $\bar{x}_1$ and $\bar{y}_1$.
We analyze the subgraph $X_{12}$ of $X$ induced by $V_{12}$.

\begin{claim}\label{symp-cl3}
  The subgraph $X_{12}$ induced by $V_{12}$ consists of a disjoint
  union of $p+1$ many $(p-1)$-size cliques,
  $C_1,C_2,\ldots,C_{p+1}$. Furthermore, the labeling procedure must
  label the vertices of each $C_i$ from the set
  $\{\alpha \bar{u}\mid \alpha\in\F^*_p\}$ of $p-1$ vectors, for some
  non-zero vector $\bar{u}\in\Span(\bar{x}_1,\bar{y}_1)$.
\end{claim}

\claimproof{It is clear from the definition of $V_{12}$ that the
  vertices in it must be labeled by nonzero vectors in
  $\Span(\bar{x}_1,\bar{y}_1)$. The proof follows easily by observing
  that a non-zero vector $\alpha\bar{x}_1+\beta\bar{y}_1$ is
  orthogonal only to itself and its scalar multiples and to no other
  vectors in $\Span(\bar{x}_1,\bar{y}_1)$.}

By the above claim, we can consistently label each vertex $u\in
V_{12}$ by a vector $\bar{u}\in\Span(\bar{x}_1,\bar{y}_1)$ assuming
$X$ is a symplectic graph.

\begin{claim}\label{symp-cl4}
  If $X=(V,E)$ is the symplectic graph of the symplectic space
  $V(2r,p)$ and $V'$ corresponds to the symplectic subspace
  $V(2r-2,p)$ spanned by $\{\bar{x}_i,\bar{y}_i\mid 2\le i\le r\}$
  then for any vertex $u\in V_{12}$ the vertex subset
  \[
  V[u]=\{v\in V\setminus V'\mid (v,u)\in E\}
  \]
has to be labeled by the vectors
\[
\{\alpha \bar{u} +\bar{v}\mid \alpha\ne 0,
\bar{v}\in\Span(\cup_{i=2}^r\{\bar{x}_i,\bar{y}_i\})\}
\]
of $V(2r,p)$.
\end{claim}

\claimproof{Suppose $v\in V\setminus V'$ is labeled by the vector
  $\bar{a}+\bar{b}$, where $\bar{a}\in\Span(\bar{x}_1,\bar{y}_1)$ and
  $\bar{b}\in V(2r-2,p)$. As $v\in V[u]$ we have
  \[
  \sympl{\bar{a}+\bar{b},\bar{u}}=0.
\]
  As $\sympl{\bar{b},\bar{u}}=0$ it implies that
  $\sympl{\bar{a},\bar{u}}=0$ which means $\bar{a}=\alpha
  \bar{u}$ for some non-zero scalar $\alpha$. Conversely, any
  vertex labeled by $\alpha \bar{u} +\bar{b}$ for $\bar{b}\in
  V(2r-2,p)$ has to lie in $V[u]$ if the labeling is consistent with
  the symplectic graph $X$.}

It follows that $|V[u]|=(p-1)\cdot p^{2r-2}$. Furthermore, if vertices
$u$ and $v$ are in the same $(p-1)$-clique in the graph $X_{12}$ then
$V[u]=V[v]$. As
$(p+1)(p-1)\cdot p^{2r-2}=p^{2r}-p^{2r-2}=|V\setminus V'|$, it follows
that the different $V[u]$ are all pairwise disjoint and form a
partition of $V\setminus V'$. Consequently, we can identify the $p+1$
many equisized vertex subsets $V[u]$ which form a partition of
$V\setminus V'$.

\subsubsection*{Labelling vertices in $V[u]$}

By Claim~\ref{symp-cl4}, the vertices of $V[u]$ have to be labeled by
vectors from $\{\alpha\bar{u}+\bar{v}\mid v\in V', \alpha\ne 0\}$.  We
will examine the adjacencies in the subgraph induced by $V[u]$ and
also the adjacencies across the different $V[u]$ and $V[u']$ to
determine a consistent labeling. Let $X[u]$ denote the subgraph of $X$
induced by $V[u]$.

For distinct vertices $v_1, v_2\in V'$ notice that
\begin{equation}\label{eqn-V[u]}
  \sympl{\bar{v}_1,\bar{v}_2} =\sympl{\alpha \bar{u}+\bar{v}_1,\beta
    \bar{u}+\bar{v}_2}.
\end{equation}
Thus, in any consistent labeling of $V[u]$ suppose
$S_{u,v_i}\subset V[u]$ are the $p-1$ vertices labeled by
$\{\alpha\bar{u}+\bar{v}_i\mid \alpha \ne 0\}$ for $i\in\{1,2\}$. Then
Equation~\ref{eqn-V[u]} implies that in the subgraph $X[u]$ the
bipartite graph induced by the subsets $S_{u,v_1}$ and $S_{u,v_2}$
is either complete or empty.

Furthermore, $\sympl{\alpha\bar{u}+\bar{v},\beta\bar{u}+\bar{v}}=0$
for all non-zero $\alpha,\beta$. Hence each $S_{u,v}, v\in V'$ induces
a $(p-1)$-clique.

How do we identify each $S_{u,v}$? Suppose $v_1,v_2\in V'$ are
distinct vertices. In $V'$ (which is already labeled by $V(2r-2,p)$)
we can find a vertex $v_3$ such that $\sympl{v_1,v_3}\ne 0$ and
$\sympl{v_2,v_3}=0$. Hence each $S_{u,v}$ can be identified because
\begin{itemize}
\item All $p-1$ vertices in $S_{u,v}$ are adjacent to each other. They
  all have identical neighborhoods in $X[u]$ by the above argument.
\item Every vertex in $V[u]\setminus S_{u,v}$ has a different
  neighborhood in $X[u]$ than the vertices in $S_{u,v}$.
\end{itemize}

Putting it together, in summary, we have shown the following.

\begin{claim}\label{symp-cl5}
  In polynomial time we can uniquely identify all the subsets
  $S_{u,v}, v\in V'$, each of size $p-1$, forming a partition of
  $V[u]$, for each $u\in V_{12}$.
\end{claim}

Thus, we have identified the subsets $S_{u,v}\subset V[u]$ for $v\in
V'$. We now need to label the elements of $S_{u,v}$ by vectors
$\{\alpha\bar{u}+\bar{v}\mid \alpha\in \F^*_p\}$.

\begin{claim}\label{symp-cl6}
Assuming $X$ is a symplectic graph, for each $v\in V'$ and $u\in
V_{12}$, we can label the $p-1$ vertices in $S_{u,v}$ by the $p-1$
vectors $\alpha\bar{u}+\bar{v}, \alpha\in\F^*_p$ arbitrarily to get a
labeling for $V\setminus V'$ that is consistent with the labeling of
$V'$.
\end{claim}

\claimproof{We first observe that adjacencies within $V[u]$ and
  between $V[u]$ and $V'$ are consistent with such a labeling by the
  arguments preceding the claim.

  We now examine the adjacencies across two disjoint $V[u]$ and
  $V[u']$. Consider $S_{u,v}\subset V[u]$ and $S_{u',v'}\subset
  V[u']$. There is an edge between these subsets if and only if for
  some non-zero scalars $\alpha$ and $\beta$ we have
  $\sympl{\alpha\bar{u}+\bar{v},\beta\bar{u'}+\bar{v'}}=0$. Now,
  $\sympl{\bar{u},\bar{u'}}\ne 0$. Hence we must have
  \[
    \alpha\beta = \frac{-\sympl{\bar{v},\bar{v'}}}{\sympl{\bar{u},\bar{u'}}}.
  \]
  Since $\alpha$ and $\beta$ are non-zero scalars, it follows that
  $\sympl{\bar{v},\bar{v'}}\ne 0$. It follows that $v$ and $v'$ being
  non-adjacent in the graph induced by $V'$ is a necessary condition
  for $S_v$ and $S_{v'}$ to have any edges between them. Furthermore,
  the above equation implies that for the vertex, say $v_\alpha\in
  S_{u,v}$ labeled $\alpha\bar{u}+\bar{v}$ there is a unique $\beta\ne
  0$ such that $v_\alpha$ is adjacent to the vertex, say $v'_\beta$ in
  $S_{u',v'}$ labeled $\beta\bar{u'}+\bar{v'}$. Hence, either there
  are no edges between $S_v$ and $S_{v'}$ (if $v$ and $v'$ are
  adjacent) or there is a perfect matching between $S_v$ and $S_{v'}$(
  if $v$ and $v'$ are non-adjacent).

  Since these are all the possible edges in $V\setminus V'$ and it is
  clearly consistent with any labeling of $S_{u,v}$ by the $p-1$
  vectors $\alpha\bar{u}+\bar{v}, \alpha\ne 0$ for all $u\in V_{12}$
  and $v\in V'$, this completes the proof of the claim.}

Putting it together, clearly if $X$ is the symplectic graph of some
non-degenerate symplectic form on $V(2r,p)$ the above algorithm
verifies that by constructing a symplectic form consistent with the
symplectic graph. The algorithm is described as a recursive procedure
below. Essentially, given a graph $X=(V,E)$ with $|V|=p^{2r}$ vertices
as input it will try to compute a labeling of $V$ by the $p^{2r}$
vectors in
$V(2r,p)=\{\sum_{i=1}^r(\alpha_i\bar{x}_i+\beta_i\bar{y}_i)\mid
\alpha_i,\beta_i\in \F_p\}$, such that the edges of $X$ are consistent
with some symplectic form on $V(2r,p)$. If $X$ is a symplectic graph
it will succeed. Otherwise, the procedure will fail and detect that
$X$ is not symplectic.

\medbreak

\noindent\textbf{Procedure Label$(X,2r,p)$}
\begin{enumerate}
\item[] \emph{The procedure takes as input a simple graph $X=(V,E)$ with
$p^{2r}$ vertices for an odd prime $p$ and computes a
labeling consistent with some symplectic form on $V(2r,p)$ if
$X$ is a symplectic graph. Otherwise, it returns ``fail''.}

\item Identify the unique dominant vertex $v_0\in V$ and label it by
  $0$ (the zero vector).

\item Pick any pair of nonadjacent vertices $x_1, y_1\in V\setminus
  \{v_0\}$ and label them by the basis vectors $\bar{x}_1$ and
  $\bar{y}_1$ respectively.

\item Let $V'$ be the set of all vertices adjacent to both $x_1$ and
  $y_1$. Check $|V'|=p^{2r-2}$. Let $X'$ be the subgraph induced by
  $V'$. Recursively call Label$(X',2r-2,p)$ and obtain a labeling of
  $V'$ by vectors in $V(2r-2,p)$ defined by the basis
  $\bar{x}_i,\bar{y}_i, 2\le i\le p$ and consistent with a symplectic
  form on $V(2r-2,p)$ (Claim \ref{symp-cl1}).

\item Determine the $p^2-1$ vertices $V_{12}$ (Claim \ref{symp-cl2})
  consisting of vertices in $V\setminus V'$ that are adjacent to all
  vertices in $V'$. Determine the partition of $V_{12}$ into $p+1$
  many $(p-1)$-cliques in the induced subgraph $G_{12}$ (Claim
  \ref{symp-cl3}).

\item From the adjacencies between $V\setminus V'$ and $V'$ determine
  the partition of $V\setminus V'$ into $p+1$ subsets $V[u]$
  (Claim \ref{symp-cl4}).

\item From the subgraphs $G[u]$ induced by the respective $V[u]$,
  determine the partition of $V[u]$ into $p-1$-subsets $S_{u,v}, v\in
  V'$ (Claim \ref{symp-cl5}).

\item Label the vertices in $S_{u,v}, v\in V'$ and $u\in V_{12}$ by
  the vectors $\alpha\bar{u}+\bar{v}, \alpha\ne 0$
  (Claim~\ref{symp-cl6}).

\item If any step above does not satisfy the stated property then
  output \emph{fail}. Else output \emph{pass} along with the complete
  labeling of $V$ by elements of $V(2r,p)$ as the witness that $X$ is
  a symplectic graph.
\end{enumerate}

\noindent\textbf{Running time analysis}~~Let $T(n)$ denote the
running time for constructing a symplectic form $\sympl{,}$ for
$V(2r,p)$ consistent with $X$ as the symplectic graph, where
$n=p^{2r}$ is the number of vertices in $X$. The inductive
construction and the rest of the computation implies the recurrence
$T(n) = O(|V|^2) +T(n-1)$, which gives an overall cubic bound
$T(n)=O(|V|^3)=O(p^{6r+3})$ on the running time.

Putting it together, we have a polynomial-time algorithm that checks
if $X$ is the symplectic graph for the symplectic space by finding
out a labeling of vertices by the vectors along with a consistent
symplectic form. \qed
\end{proof}

In summary, we have the following.

\begin{theorem}
  Given a graph $X=(V,E)$ with $p^{2r+1}$ vertices we can determine in
  polynomial time if it is the commuting graph of an extraspecial
  group of order $p^{2r+1}$ and, if so, label vertices by unique group
  elements satisfying the commuting relation.
\end{theorem}

\begin{proof}
  By Lemma~\ref{commute-to-sympl} we can obtain the graph
  $X_s=(V_s,E_s)$ from $X$ in polynomial time. By
  Theorem~\ref{ortho-recog-thm} we can check if $X_s$ is the
  symplectic graph of a symplectic form $\sympl{,}$ on $V(2n,p)$
  and also find the symplectic form.

  From the proof of Lemma~\ref{commute-to-sympl}, we have the
  following observations.
\begin{itemize}
\item The vertex of degree $|V_s|-1$ in $X_s$ is the $0$ element of
  $V(2n,p)$, and corresponds to $Z$ in $G$.
\item For each vertex $x\in V_s$ there is a closed-twins equivalence
  class of size $p-1$ containing $x$. There is a corresponding subset
  of vertices $H_x\setminus Z$ in $X$ of size $p^2-p$, where $H_x$, in
  turn, corresponds to a subgroup of $G$ of order $p^2$ that contains
  $Z$. We have a labeling of the vertex $x$ by a linear combination
  $\sum_{i=1}^n( \alpha_i \bar{x}_i + \beta_i \bar{y}_i)$, and the remaining $p-2$
  vertices in the equivalence class are labeled by non-zero scalar
  multiples of this linear combination.
\end{itemize}

  Now, using the description of the extraspecial group $G_1$ we can
  label the vertices of the clique $H_x\setminus Z$ with the group
  elements
\[
  z^j\cdot \left(\prod_{i=1}^n x_i^{\alpha_i}.y_i^{\beta_i}\right)^k,
  0\le j\le p-1, 1\le k \le p-1.
\]

 This will give us the labeling of $X$ with the elements of the
 extraspecial group $G_1$ of order $p^{2r+1}$, consistent with the
 commuting graph $X$.

 Observe that if $X_s$ is not a symplectic graph then Procedure
 Label described in Theorem~\ref{ortho-recog-thm} will detect that and
 return \emph{fail}. This completes the proof. \qed
\end{proof}

\section{A quasipolynomial time algorithm in the general case}

Assuming CFSG (the Classification of Finite Simple Groups
\cite{CFSG}), in this section we present a $2^{O(\log^3 n)}$ time
algorithm for checking if a given $n$-vertex graph is the commuting
graph of some $n$-element group and, if so, labeling the vertices of
the graph by the group elements consistent with all the commuting
pairs. For CFSG in this section we will refer to the (concrete)
constructions of finite simple groups described in Wilson's text
\cite{Wilson-book}.

\begin{enumerate}
\item Our main idea is the notion of \emph{description} $\desc(G)$ for
  any finite group $G$ that we develop in this section. It is
  essentially an encoding of the group $G$ as a \emph{binary string}
  of length at most $c\cdot \log^3 |G|$, for some fixed constant
  $c>0$, from which $G$ can be recovered efficiently. More precisely,
  we will describe an algorithm $\mathcal{D}$ that takes as input a
  binary string $w$ and, if $w=\desc(G)$ for some finite group $G$ the
  algorithm $\mathcal{D}$ will computes the multiplication table of
  $G$ as output in time polynomial in $|G|$. Otherwise, the algorithm
  will reject $w$ as an invalid description.

\item We note that this algorithm $\mathcal{D}$ can be used to
  generate the multiplication tables of all groups $G$ of order $n$ as
  follows. We enumerate all possible binary strings $w$ of length at
  most $c\cdot \log^3 n$. For each such $w$ we run the algorithm
  $\mathcal{D}$ on it. If $w=\desc(G)$ for some $G$ then $\mathcal{D}$
  will output its multiplication table. We compute $\Gamma(G)$ from
  $G$ and check if $X$ is isomorphic to $\Gamma(G)$ in time
  $2^{O(\log^3 n)}$ using Babai's graph isomorphism algorithm
  \cite{B16}.
\end{enumerate}

Thus, the main task of this section is to define $\desc(G)$ and
describe the algorithm $\mathcal{D}$. We will first define $\desc(G)$
for finite simple groups, along with how the algorithm $\mathcal{D}$
will compute the multiplication table for $G$ given $\desc(G)$ as
input. We will then extend the notion of $\desc(G)$ to all groups by
using a composition series for $G$ combined with $\desc(N)$ for each
composition factor $N$ of $G$.

\begin{remark}
  We emphasize here that our notion of ``short'' description is
  entirely different from the well-studied notion of
  generator-relator presentations of groups (see e.g.\
  \cite{Bab}). Indeed, $\desc(G)$ and its construction will be
  significantly simpler. It suffices for the purpose of recovering the
  multiplication table of the group $G$ in time polynomial in $|G|$.
\end{remark}

\subsection{Short descriptions for finite simple groups assuming CFSG}\label{cfsg-sec}

We first briefly summarize the classification theorem of finite simple
groups (CFSG)~\cite{CFSG}. A finite simple group is one of the
following:

\begin{itemize}
\item A cyclic group of prime order.
\item An alternating group $A_n, n\ge 5$.
\item A finite simple group of Lie type.
\item A sporadic finite simple group.
\end{itemize}

As there are only finite many sporadic finite simple groups, for the
algorithm $\mathcal{D}$ they can be trivially dealt with.  The cyclic
groups of prime order are also easy to deal with for the algorithm
$\mathcal{D}$ as there is a unique such subgroup of order $p$ for each
prime $p$.

\begin{lemma}\label{cyclic}
  If $G$ is cyclic of prime order $p$ then $\desc(G)$ can be defined
  as the binary encoding $\bin(p)$ of $p$ which is of length $\log p$
  bits. We can include $1$ as the generator with the additive group
  $\mathbb{Z}_p$ as the group it encodes.
\end{lemma}

\begin{proof}
  The algorithm $\mathcal{D}$ will take the binary encoding $\bin(p)$
  of $p$, verify that $p$ is prime using the polynomial-time primality
  test, and then write out the group table for the additive group of
  integers modulo $p$. \qed
\end{proof}

For several of the remaining cases, the following simple lemma gives a
framework for defining $\desc(G)$ and computing $G$'s multiplication
table from it. We use the well-known fact that any finite simple group
can be generated with two elements (see~\cite{DMT}).

\begin{lemma}\label{naive}
  Let $H$ be a finite group and $N\lhd K\le H$ be subgroups such that
  \begin{itemize}
  \item $\desc(H)$ is of length $m$.
  \item The multiplication table for $H$ can be computed from
    $\desc(H)$ in time polynomial in $|H|$.
  \item We can test membership of elements in $N$ in time polynomial
    in $|H|$.
  \end{itemize}
  Suppose $G=K/N$ is simple. Then $\desc(G)$ has size
  $O(m + \log |H|)$ and in time polynomial in $|H|$ we can compute the
  multiplication table for $G$.
\end{lemma}

\begin{proof}
  As $G=K/N$ is simple it can be generated by two elements
  $Ng_1,Ng_2\in K/N$ where $g_1,g_2\in K$. Let $\desc(G)$ be the
  indices of these two elements in $H$ along with $\desc(H)$.  By
  indices here we mean the positions of these elements in the
  multiplication table for $H$. The two indices into $H$ are of size
  $O(\log |H|)$. Therefore, $\desc(G)$ is of size $O(m + \log
  |H|)$. The claimed algorithm $\mathcal{D}$ will first use $\desc(H)$
  to compute the multiplication table for $H$. It can enumerate $H$
  and use the indices to identify the elements $g_1$ and $g_2$ in $H$.
  Using the multiplication table for $H$ it can generate the subgroup
  $K=\angle{g_1,g_2}$, listing down all its elements in time
  polynomial in $|H|$. Next, the algorithm can identify the normal
  subgroup $N$ (as given in the hypothesis) in time polynomial in
  $|H|$ and obtain the multiplication table for $G=K/N$.\qed
\end{proof}

We will apply the above lemma to the remaining cases of finite simple
groups.

\begin{lemma}\label{altern}
  If $G=A_n$ for some $n\ge 5$ then we can define $\desc(G)$
of size $O(\log |G|)$.
\end{lemma}

\begin{proof}
  We can define $\desc(A_n)$ as the encoding of two generators of
  $A_n$, by giving the index of the two generating permutations, say
  $g_1,g_2$ in $S_n$, where we consider the lexicographic order of
  enumeration of $S_n$. More precisely, each element $\pi\in S_n$ as
  an $n$-tuple $(1^\pi,2^\pi,\ldots,n^\pi)$ and the enumeration begins
  with $(1,2,\ldots,n)$ and ends with $(n,n-1,\ldots,1)$. Hence,
  $\desc(G)$ will be of size $2\log n! =O(\log |A_n|)$. We can now
  apply Lemma~\ref{naive} with $H=S_n$, $K=A_n$ and $N=\{1\}$ to
  complete the proof.\qed
\end{proof}

We now consider the finite simple groups of Lie type. It is convenient
to deal with them in two categories: classical and exceptional. We
will follow the (concrete) constructions of these groups as
explained in Wilson's text \cite{Wilson-book}.

\subsubsection{Classical Lie type}

We begin by describing a useful algorithm for efficiently enumerating
elements of the finite matrix group $\GL_n(q)$.

\begin{lemma}\label{desc-glnq}
  There is an algorithm that takes as input the pair $(n,q)$ encoded
  in binary, where $q$ is a prime power, and enumerates all elements
  of the group $\GL_n(q)$ of $n\times n$ invertible matrices over the
  field $\F_q$ in time polynomial in $|\GL_n(q)|$.
\end{lemma}

\begin{proof}
  We first recall construction of the finite field $\F_q$ given $q$ in
  binary representation. The text \cite[Chapter 25.4]{vzg-book}
  contains more details. Let $q=p^k$ for a prime $p$. Then $\F_q$ is a
  degree-$k$ extension of the prime field $\F_p$. The elements of
  $\F_p$ are essentially the residues $\{0,1,\ldots,p-1\}$ modulo the
  prime $p$, each of which can be represented in binary with at most
  $\log p$ bits. Field arithmetic in $\F_p$ is addition and
  multiplication modulo $p$. These operations can be performed using
  $O((\log p)^2)$ bit operations. The finite field $\F_q$ is described
  by a monic irreducible polynomial $g(y)\in\F_p[y]$ of degree
  $k$. Such a polynomial $g(y)$ can be found by looking through all
  the $p^k=q$ many monic degree-$k$ polynomials in $\F_p[y]$ and
  testing each for irreducibility until we find one. This algorithm is
  correct as irreducible polynomials are guaranteed to exist, and it
  has polynomial in $\log q$ running time (see \cite[Chapter
    12]{vzg-book} for details). Now, the elements of $\F_q$ can be
  represented as degree at most $k-1$ polynomials over $\F_p$ (the
  residues of polynomials in $\F_p[y]$ modulo $g(y)$). Clearly, the
  size of $g(y)=\sum_{i=0}^k \alpha_i x^i$ in binary is $O(k\log
  p)=O(\log q)$ and arithmetic operations in $\F_q$ is essentially
  addition and multiplication of degree $k-1$ polynomials in $\F_p[y]$
  modulo $g(y)$. Each such operation can be performed in $O((\log q)^3)$
  time.

  The elements of the group $\GL_n(q)$ consists of $n\times n$
  invertible matrices over $\F_q$. Each such matrix can be identified
  with the $n$-tuple of its linearly independent column vectors
  $(v_1,v_2,\ldots,v_n)$ in $\F_q$. The order of the group $\GL_n(q)$
  is given by
  \[
    |\GL_n(q)| = (q^n-1)\cdot (q^n-q)\cdots (q^n-q^i)\cdots (q^n-q^{n-1}),
  \]
  where the $(i+1)^{st}$ term $q^n-q^i$ in the product, $i\le n-1$, is
  the number of choices for the vector $v_{i+1}$ such that it is
  independent of $\{v_1,\ldots,v_i\}$. Hence, the set $L_i$ of all
  $i$-tuples $(v_1,v_2,\ldots,v_i)$ of linearly independent vectors in
  $\F_q^n$ is of cardinality $|L_i|=\prod_{j=0}^i(q^n-q^j)$, which is
  upper bounded by $|\GL_n(q)|$. The enumeration algorithm takes as
  input the pair $(n,q)$ encoded in binary (which is of size $O(\log
  n+\log q)$). It will compute the sets $L_i$ for increasing values of
  $i, 1\le i\le n$, until finally the set $L_n$ yields a listing of
  all elements of $\GL_n(q)$. To begin with, $L_1$ consists of all the
  $q^n-1$ many nonzero vectors in $\F_q^n$. Suppose now, in the
  $i^{th}$ stage, the algorithm has computed the list $L_i$. To
  compute $L_{i+1}$ the algorithm will cycle through all tuples
  $(v_1,v_2,\ldots,v_i)\in L_i$. For each such tuple it will examine
  all nonzero vectors $v\in \F_q^n$ one by one performing a linear
  independence test. If the vector $v\in \F_q^n$ is linearly
  independent of $(v_1,v_2,\ldots,v_i)$ the algorithm includes the
  $(i+1)$-tuple $(v_1,v_2,\ldots,v_i,v)$ in $L_{i+1}$. At the end of
  this stage the algorithm would have computed the set $L_{i+1}$. The
  time taken for this stage is clearly bounded by $|L_i|\cdot q^n$
  many linear independence tests, that is, checking if $v$ is
  independent of $\{v_1,v_2,\ldots,v_i\}$. Using Gaussian elimination
  \cite[Section 3.3]{schrijver-book} independence test can be done
  with $O(n^3)$ many field arithmetic operations.

  To complete the running time analysis notice that, as $q\ge 2$ and
  $n\ge 1$, $|\GL_n(q)|\ge q^n/2$. Thus, the total number of field
  arithmetic operations required over all $n$ stages is bounded by
  $O(n\cdot |\GL_n(q)|\cdot q^n)=O(n|\GL_n(q)|^2)$. As already noted,
  each field operation in $\F_q$ can be performed with $O((\log q)^3)$
  many bit operations. This completes the proof.
\end{proof}

\begin{lemma}\label{classical}
  For each finite simple group $G$ that is of classical Lie type, there
  is a description $\desc(G)$ of $O(\log |G|)$ size such that the
  multiplication table for $G$ can be computed from $\desc(G)$
  in time polynomial in $|G|$.
\end{lemma}

\begin{proof}
  These finite simple groups are categorized below, following Wilson's
  text \cite[Chapter 3]{Wilson-book}.
  \begin{enumerate}
  \item $\PSL_n(q)$ for either $n>2$ or $q>3$. These are the projective
    special linear group arising from the matrix group $\GL_n(q)$ (the
    finite group of $n\times n$ invertible matrices over the finite
    field $\F_q$).
  \item $\PSp_{2m}(q)$ which is the projective symplectic group,
    obtained as quotient of the symplectic group $\Sp_{2m}(q)$, and
    where $\Sp_{2m}(q)$ itself is a subgroup of $\GL_{2m}(q)$.
  \item $\PSU_n(q)$ which is the projective unitary group, obtained as
    a quotient of the special unitary group $\SU_n(q)$. The general
    unitary group $\GU_n(q)$ itself is a subgroup of the matrix group
    $\GL_n(q^2)$.
  \item Finally, the orthogonal groups which are also obtained
    essentially as quotients of matrix groups, but there are different
    subcases depending on both $n$ and $q$.
  \end{enumerate}

  Consider $\PSL_n(q)$. The special linear group
  $$\SL_n(q)=\{A\in\GL_n(q)\mid \det(A)=1\}$$ is a subgroup of
  $\GL_n(q)$ of index $q-1$. $\PSL_n(q)$ is obtained as the quotient
  $\SL_n(q)/Z$ where $Z$ consists of all ``scalar matrices''
  $\lambda I_n$ of determinant $1$. Thus,
  $|\PSL_n(q)|\ge |\GL_n(q)|/q^2$. We can apply Lemma~\ref{naive},
  letting $H=\GL_n(q)$, $K=\SL_n(q)$ and $N=Z$ in it.  By
  Lemma~\ref{desc-glnq}, the group $\GL_n(q)$ has a description of
  size $O(\log n + \log q)$. With that description the algorithm of
  Lemma~\ref{desc-glnq} is an algorithm will enumerate all elements of
  $\GL_n(q)$ in time polynomial in $|\GL_n(q)|$. Now, applying
  Lemma~\ref{naive} we can obtain the desired description for
  $\PSL_n(q)$.

  Next, consider $\PSp_{2m}(q)$. Recall that the symplectic group
  $\Sp_{2m}(q)$ is the subgroup of $\GL_{2m}(q)$ consisting of all
  \emph{isometries} of a non-singular alternating bilinear form on the
  vector space $\F_q^{2m}$. It turns out that $|\Sp_{2m}(q)|>q^{m^2}$.
  As $|\GL_{2m}(q)|=q^{4m^2}$, clearly
  \[
    |\GL_{2m}(q)|\le |\Sp_{2m}(q)|^4.
  \]
  Thus, the description of the group $\Sp_{2m}(q)$ is easy to
  construct.  Given $2m$ and $q$ in binary, in time polynomial in $m$
  and $q$ we can define a non-singular alternating bilinear form
  $\langle,\rangle$ on $\F_q^{2m}$. Then, as done for the previous case, we
  can enumerate the elements of $\GL_{2m}(q)$, one by one, and put each
  into $\Sp_{2m}(q)$ if it is an isometry for $\langle,\rangle$. Now, the
  simple group $\PSp_{2m}(q)$ by quotienting out the subgroup of scalar
  matrices (which has size at most $2$). Clearly, Lemma~\ref{naive}
  can be applied to obtain the desired description for $\PSp_{2m}(q)$
  along with the algorithm for obtaining its multiplication table.

  For the next case $\PSU_n(q)$, we start with the group $\GU_n(q)$
  which is the subgroup of $\GL_n(q^2)$ that preserves the
  sesquilinear form $\angle{x,y}=\sum_{i=1}^n x_i y_i^q$. From there
  on the process of obtaining $\SU_n(q)$ and then $\PSU_n(q)$ is
  exactly like in the case of $\PSL_n(q)$. In order to obtain the
  desired description for $\PSU_n(q)$ it suffices to generate
  $\GU_n(q)$ in time polynomial in $|\PSU_n(q)|$ given $n$ and
  $q$. Clearly, the entire group $\GU_n(q)$ can be listed down in time
  polynomial in $|\GL_n(q^2)|$ by enumerating, one by one, each
  $n\times n$ invertible matrix $M$ over $\F_{q^2}$ and checking if
  $M$ is in $\GU_n(q)$ (by checking if it preserves the sesquilinear
  form $\angle{x,y}$ defined above). However,
  $|\GL_n(q^2)| < |\GU_n(q)|^3$. Thus, the brute-force algorithm for
  generating all of $\GL_n(q^2)$ by looking through all $n\times n$
  invertible matrices over $\F_q$ and picking those which preserve the
  sesquilinear form will again work. Finally, the group $\PSU_n(q)$ is
  only smaller than $\GU_n(q)$ by a small power of $q$.

  The fourth category are the simple groups arising from the
  orthogonal groups. There are different cases
  to consider here.

  When $q$ is odd, there are exactly two non-singular $n$-variate
  symmetric bilinear forms over $\F_q$. If $f$ is one of them then
  $\GO_n(f)$ is the subgroup of $\GL_n(q)$ that leaves $f$ invariant.
  If $n=2m+1$ is odd, it turns out that there is only one orthogonal
  group which we may denote by $\GO_{2m+1}(q)$. We can then construct
  $\SO_{2m+1}(q)$ and $\PSO_{2m+1}(q)$ analogous to the earlier cases.
  But there is a further subgroup of index $2$ in $\PSO_{2m+1}(q)$
  denoted $\POm_{2m+1}(q)$. As before, we can apply
  Lemma~\ref{naive} to obtain the desired description of
  $\POm_{2m+1}(q)$ from which the group's multiplication table can
  be obtained in polynomial time.

  For $n=2m$, there are two orthogonal groups $\GO^+_{2m}(q)$ and
  $\GO^-_{2m}(q)$. As in the previous case we can obtain the groups
  $\POm^\epsilon_{2m}(q)$ for $\epsilon\in\{+,-1\}$. For odd $q$
  these are always simple and Lemma~\ref{naive} is again easily
  applicable.

  The characteristic $2$ case of orthogonal groups requires different
  techniques. However, Lemma~\ref{naive} is still applicable in
  defining $\desc(G)$ for the simple groups $G$ arising in this
  case. We omit the details.\qed
\end{proof}

\subsubsection{Exceptional Lie type}

We will closely follow Wilson's exposition \cite[Chapter 4]{Wilson-book}.
These finite simple groups come in ten families:
  \begin{itemize}
  \item The five families of Chevalley groups $G_2(q)$, $F_4(q)$, and $E_n(q),
    n\in\{6,7,8\}$.
  \item The two families of Steinberg--Tits--Herzig groups ${}^3{D}_4(q)$ and
${}^2{E}_6(q)$.
  \item The Suzuki and Ree groups $\Sz(2^{2n+1})$, $R_1(3^{2n+1})$ and $R_2(2^{2n+1})$ (also known as ${}^2B_2(q)$, ${}^2G_2(q)$ and ${}^2F_4(q)$ for the
appropriate $q$).
  \end{itemize}

  Unlike the classical case, these families have only one parameter
  which is the field size $q$. It turns out that the following lemma
  (analogous to Lemma~\ref{naive}) captures how these simple groups
  arise and suggest a suitable $\desc(G)$ for each $G$ along with
  the corresponding algorithm for recovering $G$ from $\desc(G)$.

  \begin{lemma}\label{except}
    Let $G$ be finite simple group from any of the ten families listed
    above. Then we can define a description $\desc(G)$ of size $O(\log q)$
    such that $G$'s multiplication table can be obtained from it in
    time polynomial in $|G|$.
  \end{lemma}

  \begin{proof}
    We will briefly describe how the groups are constructed for
    each of these ten families. The easy descriptions and algorithms
    for generating their multiplication tables will suggest themselves.
    \begin{enumerate}
    \item $\Sz(2^{2n+1})={}^2B_2(2^{2n+1})$: ~ Let $q=2^{2n+1}$. The Suzuki group
      $\Sz(q)$ (see \cite[Chapter 4.2]{Wilson-book}) can be considered as a
      subgroup of the group $\GL_4(q)$. As $\Sz(q)$ is simple, it can
      be generated with just two generators \cite{DMT}. Since
      $|\GL_4(q)|\le q^{16}$, we can describe these generators $g_1$
      and $g_2$ as two indices $i_1$ and $i_2$, respectively, each of
      size $16\log q$ bits. Define $\desc(\Sz(q))$ as the triple
      $(i_1,i_2,q)$ which is of $O(\log q)$ size with the indices and
      $q$ encoded in binary. The algorithm for recovering the
      multiplication table of $\Sz(q)$ in time polynomial in $|\Sz(q)|$
      is as follows: it first enumerates the elements of
      $\GL_4(q)$. The elements $M_1, M_2\in \GL_4(q)$ at the
      $i_1^{th}$ and $i_2^{th}$ positions in the enumeration generate
      $\Sz(q)$.  Starting with $M_1$ and $M_2$ we generate all elements
      of $\Sz(q)$ in at most $|\Sz(q)|$ rounds.  Let $S_0=\{M_1,M_2\}$.
      Inductively, suppose after $j$ rounds we have computed
      $S_j\subseteq \Sz(q)$.  Set $S_{j+1}=\{M\cdot M_i\mid i\in\{1,2\}
      \textrm{ and } M\in S_j\}$. If $S_{j+1}=S_j$ the algorithm
      outputs $S_j$ and stops. As $\angle{M_1,M_2}=\Sz(q)$ the output
      $S_j$ is $\Sz(q)$. This algorithm actually enumerates $\Sz(q)$ as
      a subset of $\GL_4(q)$ in time polynomial in $q$, and its
      multiplication table can be easily computed from this
      enumeration.
    \item $G_2(q)$:~ It turns out that $G_2(q)$ (see \cite[Chapter
      4.3]{Wilson-book}) arises as the automorphism group of the
      octonion algebra over $\F_q$. The entire octonion algebra can be
      listed in time $O(q^8)$. As the automorphisms fix the identity
      in the octonion algebra, $G_2(q)$ is essentially a subgroup of
      $\GO_7(q)$. We can enumerate $\GO_7(q)$ and pick out the
      automorphisms which will give us $G_2(q)$. As $|G_2(q)|$ is
      $q^6(q^6-1)(q^2-1)$, this will be polynomial time in the group
      size.  We can define the $\desc(G_2(q))$ of size $O(\log q)$. It
      is essentially just the encoding of $q$.
    \item $R_1(3^{2n+1})={}^2G_2(3^{2n+1})$:~This can be handled quite
     analogously to $G_2(q)$ \cite[Section 4.5]{Wilson-book}.
    \item ${}^3{D}_4(q)$:~ This group (see \cite[Chapter
      4.6]{Wilson-book}) arises as centralizer of a certain order $3$
      automorphism of $\POm^+_8(q^3)$. Because of the special nature
      of this automorphism (a product of a field automorphism and the
      triality automorphism), it has a short description of size
      $O(\log q)$. As $\POm^+_8(q^3)$ can be enumerated in time
      polynomial in $q$, we can pick out the centralizer elements in
      time polynomial in $q$.
    \item $F_4(q)$:~ when $q$ is not a power of $2$ or $3$, $F_4(q)$
      arises as automorphisms of the finite (non-associative) Albert
      algebra which is $27$-dimensional \cite[Chapter
        4.8]{Wilson-book}. Thus, $F_4(q)$ can be picked out from
      $\GL_{27}(q)$ by enumerating it and checking each element for
      being an automorphism. The finite Albert algebra can be defined
      in $O(\log q)$ size by defining the multiplication of its basis
      elements.
    \item $R_2(2^{2n+1})={}^2F_4(2^{2n+1})$:~ This case is quite similar to
      the previous one \cite[Section 4.9]{Wilson-book}.
    \item $E_6(q)$:~ This group can be defined as automorphisms of a
      certain trilinear form over the Albert algebra \cite[Chapter
        4.10]{Wilson-book}. After quotienting out the scalar subgroup
      from this automorphism group we will obtain $E_6(q)$. The same
      algorithm of enumerating the bigger group and selecting the
      automorphisms and then quotienting out the scalars will work.
    \item $E_7(q)$ and $E_8(q)$:~ Both these arise as automorphisms of
      a finite Lie algebra \cite[Chapter 4.1  non-zero vector $\bar{u}\in\Span(\bar{x}_1,\bar{y}_1)$.
2]{Wilson-book}. For
      $E_8(q)$ we have to consider a $248$-dimensional Lie
      algebra. Thus we can look for $E_8(q)$ inside $\GL_{248}(q)$ by
      enumerating all its elements and picking out the
      automorphisms. All we need is that from a $O(\log q)$ size
      description we should be able to obtain the Lie algebra. For
      that we only need to express the multiplication of basis
      elements as an $\F_q$ linear combination of the basis (which is
      of size $248$). $E_7(q)$ is similarly dealt with as
      automorphisms of a smaller dimensional Lie algebra.\qed
    \end{enumerate}

  \end{proof}

  The above lemmas immediately imply the following.

  \begin{theorem}\label{desc-simple}
    There is an algorithm that, given as input $\desc(G)$ for a
    finite simple group $G$, outputs its multiplication table
    in time polynomial in $|G|$.
   \end{theorem}

\subsection{Short descriptions for finite groups}
We first show that every finite group $G$ has a description $\desc(G)$
of bit-length $O((\log|G|)^3)$.  Then we describe the algorithm that
given as input a bit string $y$, if $y=\desc(G)$ for some finite group
$G$ then the algorithm runs in time $2^{O((\log|G|)^3)}$ and outputs the
multiplication table of a group $G'$ isomorphic to $G$.

Building on the construction of $\desc(G)$ for finite simple groups
$G$, described in Section~\ref{cfsg-sec}, we obtain $\desc(G)$ for all
finite groups with a straightforward construction that uses a
composition series for $G$. It is convenient to use
Erd\H{o}s-R\'{e}nyi cube generating sequences \cite{ER}.

\begin{defn}[Erd\H{o}s-R\'{e}nyi cube generating sequences]
  Let $G$ be a finite group. A \emph{cube generating sequence}
  (abbreviated as c.g.s.) for $G$ is a sequence $(g_1,g_2,\ldots,g_k)$
  of elements from $G$ such that
  \[
  G=\{g_1^{\epsilon_1}g_2^{\epsilon_2}\cdots g_k^{\epsilon_k}\mid \epsilon_i\in\{0,1\}
  \hbox{ for each } i\}.
  \]
Thus, given the cube generating sequence $(g_1,g_2,\ldots,g_k)$, each
element of $G$ can be defined by the bit vector
$(\epsilon_1,\epsilon_2,\ldots,\epsilon_k)$. We note that this bit
vector need not be unique for a group element.
\end{defn}

\begin{theorem}\label{ER}{\rm\cite{ER}}
Every finite group $G$ has a cube generating sequence of size at most $2\log
|G|$.
\end{theorem}

Let $X=(g_1,g_2,\ldots,g_k)$ be cube generating sequence for a group
$G$. Its length $k$ is denoted $|X|$. Abusing notation, by $g\in X$ we
will mean that $g=g_i$ for some $i$.

We note that cube generating sequences can be computed in polynomial
time.

\begin{lemma}\label{cgs-find}
  There is a polynomial-time algorithm that takes a group $G$ by its
  multiplication table as input and outputs a cube-generating sequence
  of length at most $2\log |G|$ for $G$.
\end{lemma}

\begin{proof}
  The simple algorithm is directly based on a ``cube doubling'' trick
  due to Babai and Szemer\'{e}di \cite{BSz84}. Suppose we have picked
  the sequence $(g_1,g_2,\ldots,g_i)$ of group elements from $G$. Let
  $C_i$ be the cube it generates. That is,
\[
C_i=\{g_1^{\epsilon_1}g_2^{\epsilon_2}\cdots g_i^{\epsilon_i}\mid \epsilon_j\in\{0,1\}
\hbox{ for } 1\le j \le i\}.
\]
For any $g\in G$ let $C_ig=\{xg\mid x\in C_i\}$. As $|C_i|\le 2^i$ and
each element in $C_i$ can be computed with $i-1$ group
multiplications, the set $C_i$ can be enumerated with at most
$(i-1)2^i$ many group multiplications. We can enumerate $C_ig$ with
$2^i$ additional group multiplications.

Now, if $C_ig\cap C_i\ne \emptyset$ for $g\in G$ then $g$ is in the
cube generated by the sequence
$(g_i^{-1},g_{i-1}^{-1},\ldots,g_1^{-1},g_1,\ldots,g_i)$. Thus, if for all
$g\in G$ we have $C_ig\cap C_i\ne \emptyset$ then we have a cube generating
sequence of length $2i$. Otherwise, there is a $g_{i+1}\in G$ such
that
\[
C_ig_{i+1}\cap C_i = \emptyset,
\]
and the cube $C_{i+1}$ generated by $(g_1,g_2,\ldots,g_{i+1})$ is
twice the cardinality of $C_i$ by a simple counting argument. Clearly,
we can find $g_{i+1}$ in time polynomial in $|G|$ by searching through
$G$.

As $|C_{i+1}|=2\cdot |C_i|$, it follows that this procedure will
eventually find a cube generating sequence $C_k$ for some
$k\le 2\log|G|$ with at most polynomial in $|G|$ group multiplications
and hence in time polynomial in $|G|$.  \qed
\end{proof}

We now explain the construction of $\desc(G)$ for a finite group $G$,
where $G$ is given by its multiplication table.  Compute a composition
series for $G$ (which can be done in time polynomial in $|G|$ by
computing normal closures of elements):
  \[
1=N_0\lhd N_1\lhd \cdots \lhd N_r=G.
\]
For each quotient group $H_i=N_i/N_{i-1}$ for $1\le i\le r$, we can
compute a cube generating sequence
$Y_i=(N_{i-1}g_{i,1},N_{i-1}g_{i,2},\ldots,N_{i-1}g_{i,k_i})$ of length
$k_i\le 2\log |H_i|$ in polynomial time, by Lemma~\ref{cgs-find}.  For
each $i$, let $X_i=(g_{i,1},g_{i,2},\ldots,g_{i,k_i})$ be the
corresponding sequence of coset representatives.

\begin{prop}
  The sequence $X=(X_1, X_2, \cdots, X_i)$ obtained by concatenating
  the sequences $X_j, 1\le j\le i$, defined above, in that order is a
  cube generating sequence for the subgroup $N_i$ of length at most
  $\sum_{j=1}^i 2\log |H_j|=2\log |N_i|$, for $1\le i\le r$. In
  particular, $X=(X_1, X_2, \cdots, X_r)$ is a cube generating
  sequence for $G$ of length at most $2\log |G|$.
\end{prop}

\begin{proof}

  Suppose $H\lhd G$ and $Y=\{y_1,y_2,\ldots,y_\ell\}$ is a cube
  generating sequence for $H$ and $Z=\{Hz_1,Hz_2,\ldots,Hz_m\}$ is a
  cube generating sequence for $G/H$. We note that any element
  $Hg\in G/H$ can be expressed as
\[
Hg =  Hz_1^{\epsilon_1}z_2^{\epsilon_2}\cdots z_m^{\epsilon_m},
  \epsilon_i\in\{0,1\},
\]
which implies that any element of $G$ can be expressed as
\[y_1^{\delta_1}y_2^{\delta_2}\cdots
y_\ell^{\delta_\ell}z_1^{\epsilon_1}z_2^{\epsilon_2}\cdots
z_m^{\epsilon_m},\, \epsilon_i,\delta_j\in\{0,1\}.\]
Now, building on this, an easy inductive argument yields the proposition.\qed
\end{proof}

For $1\le i\le r$, let $Y_i=(X_1,X_2,\ldots,X_i)$. By the above
proposition, $Y_i$ is a cube generating sequence for the group $N_i$
of length at most $2\log |N_i|$. As $N_{i-1}\lhd N_i$, for each
$x\in X_i$ we have $xN_{i-1}=N_{i-1}x$, $1<i\le r$. In particular, for
$y\in Y_i$ there is a unique element $w(x,y)\in N_{i-1}$ such that
we have
\begin{equation}\label{eqn-rel}
xy=w(x,y)x,
\end{equation}
where $w(x,y)$ is a word over the cube generating sequence
$Y_{i-1}$. Thus, $w(x,y)$ is a word of length at most $2\log
|N_{i-1}|$, expressed using the c.g.s.\ $Y_{i-1}$. Let $R_i$ consist
of the set of all such equations defined by Equation~(\ref{eqn-rel})
for $y\in Y_{i-1}$ and $x\in X_i$.  Notice that $R_i$ has
$|Y_{i-1}|\cdot |X_i|$ equations in it and each word $w(x,y)$ over
$Y_{i-1}$ is defined by the bit vector of exponents in the cube
product over $Y_{i-1}$. Thus, the size of $R_i$ is
\[
\size(R_i) =
O(|Y_{i-1}|\cdot |X_i|\cdot |Y_{i-1}|)=O((\log |N_{i-1}|)^2\cdot \log
|H_i|).
\]

We are now ready to define $\desc(G)$ for the finite group $G$. It is
defined as an efficient\footnote{We can encode an $r$-tuple of
  $3$-tuples in binary by using some standard encoding which is
  polynomial time computable and its inverse is also polynomial time
  computable.}
  \emph{binary encoding} of the $r$-tuple
defined below
\[
\desc(G) ~=~ [(\desc(H_1),X_1,R_1),(\desc(H_2),X_2,R_2),\ldots,(\desc(H_r),X_r,R_r)],
\]
where $R_1$ is the empty set. The size of $\desc(G)$ is clearly
\[
O(\sum_{i=1}^r \log~|H_i|) + O(\sum_{i=1}^r \size(X_i)) + O(\sum_{i=1}^r \size(R_i)),
\]
where $\size(X_i)$ and $\size(R_i)$ denote the bit length of their
encodings described above. Notice that
$\sum_{i=1}^r\log~|H_i| =\log |G|$. It follows that
$\sum_{i=1}^r\size(X_i)\le O((\log|G|)^2)$, because
$|X_i|\le 2\log |H_i|\le 2\log |G|$ and each element of $X_i$ requires
at most $\log |H_i|$ bits to encode it. We now bound
$\sum_{i=1}^r \size(R_i)$.

\begin{eqnarray*}
  \sum_{i=1}^r \size(R_i) &=& O(\sum_{i=1}^r (\log^2|N_{i-1}|\cdot\log |H_i|)\\
  &\le& \sum_{i=1}^rO( \log^2|G|)\cdot\log~|H_i|\\
  &=& O((\log|G|i)^2)\cdot \sum_{i=1}^r \log~|H_i|\\
  &=& O((\log^3|G|)^2).
\end{eqnarray*}

The above discussion has the following immediate consequence.

\begin{lemma}
  Given as input the multiplication table of a finite group $G$, there
  is an algorithm that computes in time polynomial in $|G|$ its
  description $\desc(G)$.
\end{lemma}

However, we are interested in the converse statement, about generating
$G$ from its description $\desc(G)$ in time polynomial in $|G|$.  The
next theorem, which shows this, will allow us to enumerate the
multiplication tables of all groups of order $n$ up to isomorphism in
time $2^{O((\log n)^3)}$.

\begin{theorem}\label{list-groups}
  There is an algorithm that given as input $\desc(G)$, for a finite
  group $G$, outputs its multiplication table in time polynomial in
  $|G|$. As a consequence, there is an algorithm that outputs the list
  of multiplication tables of all distinct groups of order $n$ in time
  $2^{O((\log n)^3)}$.
\end{theorem}

\begin{proof}
  As argued above, every finite group $G$ has a description
  $\desc(G) =
  [(\desc(H_1),X_1,R_1),(\desc(H_2),X_2,R_2),\ldots,(\desc(H_r),X_r,R_r)]$
  where the $r$-tuple is encoded in a binary string of length bounded
  by $c\cdot \log^3|G|$ for some constant $c>0$ (which is independent
  of $G$).

  Let $w$ be a binary string. If it encodes $\desc(G)$ for some group
  $G$ such that $|G|=n$, then
\begin{itemize}
\item  $|w|\le c\log^3 n$ and $w$ decodes to an $r$-tuple of $3$-tuples
 \[[(\desc(H_1),X_1,R_1),(\desc(H_2),X_2,R_2),\ldots,(\desc(H_r),X_r,R_r)].\]
\item Moreover, for some composition series
  $1=N_0\lhd N_1\lhd \cdots \lhd N_r=G$ of $G$,
  $\desc(N_i)$ is the binary encoding of the $i$-tuple
  \[
    \desc(N_i) =
    [(\desc(H_1),X_1,R_1),(\desc(H_2),X_2,R_2),\ldots,(\desc(H_i),X_i,R_i)],
  \]
  for $1\le i\le r$, where $H_i=N_i/N_{i-1}$.
\end{itemize}

  The algorithm will compute the multiplication table for $N_i$,
  inductively for increasing values of $i, 1\le i\le r$. For $i=1$,
  it can be obtained from $\desc(H_1)$ by Theorem~\ref{desc-simple}.

  Given the multiplication table for $N_{i-1}$ and the triple
  $(\desc(H_i),X_i,R_i)$ we first obtain $H_i$ by
  Theorem~\ref{desc-simple}, and the actual elements of the c.g.s.\
  $X_i=(x_1,x_2,\ldots,x_\ell)$ for $\ell\le 2\log |H_i|$. The c.g.s.\
  for $N_{i-1}$ is $Y_i=(X_1,X_2,\ldots,X_{i-1})=(y_1,y_2,\ldots,y_k)$
  for some $k\le 2\log |N_{i-1}|$. Applying Equation~(\ref{eqn-rel}) we
  have as part of $R_i$
  \begin{equation}\label{eqn-rel1}
    x_\mu y_\nu=w(x_\mu,y_\nu)x_\mu,
  \end{equation}
  for each $\mu\in [\ell]$ and each $\nu\in [k]$. Here, note that
  $w(x_\mu,y_\nu)$ is expressed as a product $y_1^{e_1}y_2^{e_2}\cdots
  y_k^{e_k}$ for $(e_1,e_2,\ldots,e_k) \in \{0,1\}^k$.

  We now explain how to obtain the multiplication table for the
  subgroup $N_i$ from the above data in time polynomial in $|N_i|$.
  More precisely, the data comprises of the multiplication table for
  $N_{i-1}$, the multiplication table for the quotient group
  $H_i=N_i/N_{i-1}$ and $R_i$ (Equation~(\ref{eqn-rel1}) for all $y_\nu$
  and $x_\mu$).

We first note that $N_i = \cup_x N_{i-1}x$, where $x$ is in the cube
\[C=\{\prod_{i=1}^\ell x_j^{\varepsilon_j}\mid \varepsilon_j\in
\{0,1\}\textrm{ and } x_j\in X_i\}.\]
Thus, $N_i=N_{i-1}C$. So, any element in $N_i$ is of the
form $ab, a\in N_{i-1}, b\in C$. We need to express the product of any
two elements $a_1b_1, a_2b_2\in N_i$, where $a_1,a_2\in N_{i-1}$ and
$b_1,b_2\in C$ in the form $a_3b_3, a_3\in N_{i-1}, b_3\in C$.

\begin{claim}
  For any $x_\mu\in X_i$ and $a\in N_{i-1}$ we can find $a'\in N_{i-1}$
  such that $x_\mu a = a' x_\mu$ in time polynomial in $|N_{i-1}|$.
\end{claim}

\begin{poc}
We first express $a$ using the c.g.s.\ $Y_i$ for
  $N_{i-1}$ to write it as
  \[
  a=y_1^{\beta_1}y_2^{\beta_2}\cdots y_\ell^{\beta_\ell}.
  \]
  Then we can use Equation~\eqref{eqn-rel1} $\ell$ times to write $x_\mu
  a = x_\mu \prod_{j=1}^\ell y_j^{\beta_j}$ as the product $x_\mu a =
  (\prod_{j=1}^\ell w(x_\mu,y_j)^{\beta_j})x_\mu$. We can compute $a'=
  \prod_{j=1}^\ell w(x_\mu,y_j)^{\beta_j}$ using the multiplication
  table for $N_{i-1}$ as each $w(x_\mu,y_j)$ is in $N_{i-1}$.
\end{poc}

Using the above claim repeatedly we can express products of the form
$x a$ for $x\in C, a\in N_{i-1}$ as $a' x$ for some $a'\in
N_{i-1}$. Finally, we can express the product of any two elements
$a_1b_1, a_2b_2\in N_i$, where $a_1,a_2\in N_{i-1}$ and $b_1,b_2\in C$
in the form $a_3b_3, a_3\in N_{i-1}, b_3\in C$.  The algorithm runs in
time polynomial in $|N_{i-1}|$ and $|X_i|$ and hence in time
polynomial in $|N_i|$.

Applied for increasing values of $i$, finally, for $i=r$ we obtain the
table for $N_r=G$. The overall algorithm clearly runs in polynomial in
$|G|$ time.

The algorithm for enumerating all order $n$ group multiplication
tables is easy to obtain from the above. We run through all strings
$w$ of length at most $c\log^3 n$, where $n=|G|$, and apply the above algorithm
to generate the multiplication for $G$ if $w$ encodes $\desc(G)$.

Notice that this enumeration may contain duplicates because the same
group can have different composition series and hence different
descriptions. We can eliminate the duplicates by running Miller's
group isomorphism test \cite{Mi78}. Each such test takes $n^{\log
  n}$. Since there are at most $2^{\log^3 n}$ groups being listed the
running time to detect multiplicities takes time at most $2^{\log^3
  n}\cdot n^{\log n}\cdot 2^{\log^3 n}=2^{O((\log n)^3)}$.\qed
\end{proof}

\begin{remark}
Some notes concerning the preceding results:
\begin{enumerate}
\item Computing a composition series for $G$ can be done in
  time polynomial in $|G|$ by computing normal closures of elements:
  \[
1=N_0\rhd N_1\rhd \cdots \rhd N_r=G.
\]
More precisely, for each $g\in G$ we compute its normal closure by
starting with the cyclic subgroup $N=\angle{g}$ and repeatedly
replacing it with the (larger) subgroup generated by the set
$\{xyx^{-1}\mid x\in G, y\in N\}$ until $N$ does not grow any further.
By construction, the final $N$ is a normal subgroup of $G$ and its
computation takes time polynomial in $|G|$. If $G$ is not simple, then
for some $g\in G$ this construction will yield a non-trivial normal
subgroup $N$. We can now recursively find composition series for $N$
and $G/N$ (we have multiplication tables for both) and put them
together to obtain a composition series for $G$ in time polynomial in
$|G|$.


\item For each group $N_i$ we compute a cube generating sequence $X_i$
  of size at most $k_i\le 2\log |N_i|$. This computation is done in
  time polynomial in $|N_i|$, by Lemma~\ref{cgs-find}. As part of the
  description  of $G$, for each $1\le i\le r$ we obtain a set of
  relations $R_i$ as defined in Equation~(\ref{eqn-rel}).
\end{enumerate}

Our algorithm described in Theorem~\ref{list-groups} is broadly based
on a result of McIver and Neumann \cite{MN87} that bounds the number
of $n$ element groups by $2^{O((\log n)^3)}$. There is related work by
Babai \textit{et al.}~\cite{Bab} which shows that all simple groups of
$n$ elements (except one family called the Ree groups) have
generator-relator presentations of size $O((\log n)^2)$.
\end{remark}

\begin{theorem}\label{quasipoly-thm}
There is a $2^{O((\log n)^3)}$ time algorithm that recognizes if a given
$n$-vertex graph is the commuting graph of a group of order $n$.
\end{theorem}

\begin{proof}
  Let $X$ be an $n$-vertex graph. The algorithm cycles through all
  multiplication tables of distinct groups of order $n$, listed by the
  $2^{O((\log n)^3)}$ time algorithm described in
  Theorem~\ref{list-groups}. For each group $G$ that is listed we can
  compute its commuting graph $\Gamma(G)$ in polynomial time.  We can
  test if $\Gamma(G)$ is isomorphic to the input graph $X$ using
  Babai's $2^{O((\log n)^3)}$ time algorithm \cite{B16}.

Thus, the overall computation takes $2^{O((\log n)^3)}$ time.\qed
\end{proof}

The algorithm of Theorem~\ref{quasipoly-thm} does not reveal any new
relationship between the group and its commuting graph. It would be
more interesting to obtain an algorithm that exploits the commuting
graph structure.

\section{Isoclinism of groups and commuting graphs}\label{sec-isoclin}

Motivated by the polynomial-time algorithm for recognizing commuting
graphs of extraspecial groups, in this section we briefly discuss
isoclinism of groups and its connection to commuting graphs.

The \emph{lexicographic product} of graphs $\Gamma_1$ and $\Gamma_2$ is the
graph on the vertex set $V(\Gamma_1)\times V(\Gamma_2)$, in which $(v_1,w_1)$
and $(v_2,w_2)$ are joined if and only if either $v_1$ and $v_2$ are joined in
$\Gamma_1$, or $v_1=v_2$ and $w_1$ and $w_2$ are joined in $\Gamma_2$. If
$\Gamma_2$ is a complete graph, then the lexicographic product is obtained
from $\Gamma_1$ by blowing up each vertex to a complete graph on
$|V\Gamma_2|$ vertices.

Commutation in a group $G$ is a map from $G\times G$ to the derived group $G'$;
multiplying the arguments by elements of $Z(G)$ does not change the function
value, so commutation induces a map from $G/Z(G)\times G/Z(G)$ to $G'$.
We say that groups $G$ and $H$ are \emph{isoclinic} if there are isomorphisms
$\theta:G/Z(G)\to H/Z(H)$ and $\phi:G'\to H'$ which intertwine the commutation
map; that is, for all $g_1,g_2\in G$,
\[[\theta(g_1Z(G)),\theta(g_2Z(G))]=\phi([g_1Z(G),g_2Z(G)]).\]

Now the commuting graph of a group $G$ is the lexicographic product of a
graph on $G/Z(G)$ with $g_1Z(G)$ and $g_2Z(G)$ joined if $[g_1,g_2]=1$
(this graph can also be described as the induced subgraph of the commuting
graph of $G$ on a transversal for $Z(G)$ in $G$)
with a complete graph on $Z(G)$; the graph on $G/Z(G)$ is determined by the
isoclinism type. Hence, as observed by Pezzott and Nakaoka~\cite{pn}, we have:

\begin{theorem}\label{thm:isoclin}
Isoclinic groups of the same order have isomorphic commuting graphs.
\end{theorem}

\begin{remark}
  It follows from Lemma~\ref{lem:iso} that the extraspecial groups
  $G_1$ and $G_2$, defined in Section~\ref{extrasp-sec}, are
  isoclinic. Indeed, the recognition algorithm of
  Section~\ref{extrasp-sec} implicitly uses this property.
\end{remark}

Is there a converse to Theorem~\ref{thm:isoclin}? That is, if two
groups have isomorphic commuting graphs, must they be isoclinic? This
holds for various classes of groups, such as abelian groups,
non-abelian simple groups, and extraspecial groups (as we have seen).

However, it is not true in general; we recycle an example taken from
\cite{CK} to show this. Let $G$ be the group of order~$64$ which is
\texttt{SmallGroup(64,182)} in the {\sc SmallGroups} library \cite{Be02} in
\textsf{GAP} \cite{gap}. A \emph{Schur cover} of $G$ is a group $H$ of maximal
order subject to having a subgroup $Z\le H'\cap Z(H)$ such that
$H/Z\cong G$.  The Schur multiplier of $G$ has order $2$, so any Schur
cover of $G$ has order $128$. Moreover, the Bogomolov multiplier (see
\cite{CK} for context and other references) of $G$ is equal to the
Schur multiplier, which implies that any Schur cover $H$ is
\emph{commutation-preserving}: that is, two elements $a,b\in H$
commute if and only if their projections $Za,Zb\in G$ commute. This
implies that the commuting graph of a Schur cover is obtained from the
commuting graph of $G$ by replacing each vertex with a clique of size
$2$, with all edges between cliques corresponding to adjacent
vertices. This procedure also describes the commuting graph of
$G\times C_2$. On the other hand, it is easy to verify computationally
that the derived groups of $H$ and $G\times C_2$ are not isomorphic,
so these groups cannot be isoclinic.

We note that the group $G$ has nilpotency class $3$, as do all of its
Schur covers (these are \texttt{SmallGroup(128,$i$)} for $i\in\{789,
790, 791, 815, 816, 817\}$ in the \textsf{GAP} library). This suggested that it
might be true that groups with nilpotency class $2$ and isomorphic commuting
graphs are isoclinic.
However, Eamonn O'Brien (personal communication) has shown that even this is
false. The groups \texttt{SmallGroup(64,18)} and \texttt{SmallGroup(64,215)}
have nilpotency class~$2$ and isomorphic commuting graphs, but are not
isoclinic (their derived quotients have exponent $4$ and $2$ respectively).
He also has counterexamples with orders $3^6$ and $5^6$.

\begin{ques}
What can be said about pairs of groups with isomorphic commuting graphs? Do
non-isoclinic examples exist for all orders $p^6$ with $p$ prime?
\end{ques}

\part{Forbidden subgraphs}

\section{Preliminaries}\label{ten-sec}

A number of important graph classes can be defined either structurally or in terms of forbidden induced subgraphs. Here are structural descriptions.
\begin{enumerate}
\item
$\Gamma$ is called a {\em cograph} if it can be built from $1$-vertex graphs
by the operations of disjoint union and complementation.
\item
$\Gamma$ is called a {\em chordal graph} provided that $\Gamma$ contains no induced cycles of length greater than $3$; in other words,
any cycle of length at least $4$ contains a chord.
\end{enumerate}

The characterization by forbidden subgraphs is a composite of results from
\cite{FH,ChHa,Sumner,CRST}.

\begin{prop}\label{t:forbid}
\begin{enumerate}
\item A graph is a cograph if and only if it forbids $P_4$.
\item A graph is chordal if and only if it forbids $\{C_n:n\ge4\}$.
\end{enumerate}
\end{prop}

Since a cycle of length $5$ or more contains a $4$-vertex path, we have the following observation.

\begin{obs}\label{n-obs1}
If a graph is $\{P_4,C_4\}$-free, then it is chordal.
\end{obs}

We conclude this section with a general remark.
In the commuting graph $\Gamma(G)$ of a group $G$, elements of the centre
$Z(G)$ are dominating vertices; this is the reason why, to investigate questions
about connectedness, they are excluded. However, in our case,
it doesn't matter:

\begin{prop}\label{cog-subg}
\begin{enumerate}
\item The class of groups whose commuting graph is $\mathcal{F}$-free is
subgroup-closed.
\item If no graph in $\mathcal{F}$ has a dominating vertex, then the class of
groups with $\mathcal{F}$-free commuting graph is independent of whether
vertices in $Z(G)$ are included or excluded.
\end{enumerate}
\end{prop}

\begin{proof}
(a) This is immediate from the fact that, if $H$ is a subgroup of $G$, then
$\Gamma(H)$ is an induced subgraph of $\Gamma(G)$. (We remark that the same
result holds for various other graphs such as the power graph, enhanced power
graph, nilpotency and solvability graphs. But it does not hold for yet others
such as the generating graph or the deep commuting graph.)

\smallskip

(b) If $X$ is an induced subgraph of $\Gamma(G)$, then vertices of
$X\cap Z(G)$ would be dominating vertices in $X$. \qed
\end{proof}

In the cases we are interested in (chordal graphs and
cographs), the excluded subgraphs are paths and cycles with more than three
vertices. None of these have dominating vertices, so we can take the
vertex set of the commuting graph to be all of $G$, but to decide if it has
any of these properties we can delete vertices in $Z(G)$.

\medskip

Also to prove main results of this paper, we need the following
assertion which generalizes a result obtained by
Wielandt~\cite{wielandt} in 1939.

\begin{prop}\label{Wielandt} Let $G$ be a finite group. Assume that one of the following statements holds:
\begin{enumerate}
\item $Z(G)=1$;
\item $G=G'$.
\end{enumerate}
Then $C_{\Aut(G)}(\Inn(G))=1$.
\end{prop}

\begin{proof}

  Let $\phi$ be an automorphism of $G$ which commutes with each inner
  automorphism of $G$. Then for each $g, h \in G$, we
  have $$\phi(h)^g=\phi(h^g)=\phi(h)^{\phi(g)},$$
  i.\,e. $[g,\phi]=g^{-1}\phi(g) \in Z(G)$ and $\phi(g)=g\cdot z_g$
  for some $z_g \in Z(G)$ (depending of $g$).

If $Z(G)=1$, then $z_g=1$ for each $g$, therefore $\phi$ is trivial and $C_{\Aut(G)}(\Inn(G))=1$, as required.

Now assume that $G$ is perfect, i.e. $G'=G$. Note that $\phi([x,y])=[\phi(x),\phi(y)]=[x,y]$ for each $x, y \in G$. Since $G$ is generated by commutators, we have $z_g=1$ for each $g$. Thus, again $\phi$ is trivial and so, $C_{\Aut(G)}(\Inn(G))=1$. \qed
\end{proof}

\section{Structure of groups whose commuting graphs are cographs or chordal graphs}\label{Structure}

In this section, we will investigate the group-theoretical structure
of a group whose commuting graph is a cograph or a chordal graph.  In
particular, in some sense, the class of such groups generalizes the
class of finite solvable groups.

First, we investigate special cases of groups like direct and central
products and Frobenius groups whose commuting graph is a cograph or a
chordal graph, and then prove the general structure theorem.

\subsection{Direct and central products}

\begin{obs}\label{new-ob-cc}
Let $A$ be an abelian group. Then $\Gamma(A\times G)$ is a cograph {\rm(}resp. a chordal graph{\rm)} if and only if $\Gamma(G)$  is a cograph {\rm(}resp. a chordal graph{\rm)}.
\end{obs}

\begin{lemma}\label{zj-cg}
Let $H$ and $K$ be non-abelian groups.
Then $\Gamma(H\times K)$ has an induced subgraph isomorphic to $P_4$ and an induced subgraph isomorphic to $C_4$.
\end{lemma}

\begin{proof}
Let $a$ and $b$ be non-commuting elements of $G$, and $u$ and $v$
non-commuting elements of $H$. Then
$$
(a,u)\sim (e_H,u) \sim (b,e_K) \sim (e_H,v)
$$
is an induced path with four vertices, and
$$
(a,e_K)\sim (e_H,u) \sim (b,e_K) \sim (e_H,v) \sim (a,e_K)
$$
is an induced cycle of length $4$, as desired. \qed
\end{proof}

By Lemma~\ref{zj-cg} and Observation~\ref{new-ob-cc}, the following result holds.

\begin{prop}\label{zj-thm}
Let $H$ and $K$ be groups. Then
\begin{enumerate}
\item $\Gamma(H\times K)$ is  a cograph if and only if one of $H$ and $K$ is abelian and the commuting graph of the other is $P_4$-free;
\item $\Gamma(H\times K)$ is chordal if and only if one of $H$ and $K$ is abelian and the other is $C_n$-free for all $n\ge 4$.
\end{enumerate}
\end{prop}

Recall that a finite nilpotent group is the direct product of its Sylow subgroups.  Thus, Proposition~\ref{zj-thm} implies the following result.

\begin{cor}\label{Nilp-cog}
Let $G$ be a non-abelian nilpotent group. Then $\Gamma(G)$ is a
cograph {\rm(}resp. a chordal graph{\rm)} if and only if there exists exactly one non-abelian Sylow subgroup in $G$ whose commuting graph is $P_4$-free {\rm(}resp. $C_n$-free for all $n\ge 4${\rm)}, while the other Sylow subgroups are abelian.
\end{cor}

We saw in the first part the definition of the strong product of graphs and
the fact that the commuting graph of $G\times H$ is
$\Gamma(G)\boxtimes\Gamma(H)$ (Proposition~\ref{cgraph-strong}).

We can extend this to central products as follows. Let $\mathfrak{N}_2$ be the class of all finite nilpotent groups of nilpotency class at most~$2$ (thus, the class consists of finite abelian groups and finite non-abelian groups of nilpotency class~$2$).

\begin{lemma}\label{lem_cp}
Let $H$ and $K$ be two groups not in the class $\mathfrak{N}_2$.
Then $\Gamma(H\circ K)$ is neither a cograph nor a chordal graph.
\end{lemma}

\begin{proof}
By assumption, we can choose $a,b\in H$ such that the commutator
$[a,b]$ is not in the centre of $H$, and similarly $u,v\in K$.
Then
\[(a,u)\sim(a,e_K)\sim(e_H,v)\sim(b,v)\]
carries a subgroup $P_4$ of $\Gamma(H\times K)$ such that the projection to
$H\circ K$ is an isomorphism, and also
\[(a,e_K)\sim(e_H,u)\sim(b,e_K)\sim(e_H,v)\sim(a,e_K)\]
is a subgraph isomorphic to $C_4$ which projects isomorphically to $H\circ K$.
\qed
\end{proof}

\subsection{Frobenius groups}

In this subsection we investigate when the commuting graph of a Frobenius group is a cograph or a chordal graph. We use the structure of the commuting graphs of
Frobenius groups established in Theorem~\ref{frob-thm}. We prove the following
assertion.

\begin{prop}\label{Frobenius} Let $G$ be a Frobenius group with Frobenius complement $H$ and Frobenius kernel $K$ and $\mathcal{F}$ be a collection of graphs such that no graph in $\mathcal{F}$ has a dominating vertex. Then $\Gamma(G)$ is $\mathcal{F}$-free if and only if $\Gamma(K)$ is $\mathcal{F}$-free and $\Gamma(H)$ is $\mathcal{F}$-free.
\end{prop}

\begin{proof} This is clear from our remarks about dominating vertices and the
fact (observed in Theorem~\ref{frob-thm}) that $\Gamma(G\setminus\{1\})$ is the
disjoint union of $\Gamma(H\setminus\{1\})$ and $|H|$ copies of
$\Gamma(K\setminus\{1\})$.
\end{proof}

\begin{cor}  Let $G$ be a Frobenius group with Frobenius complement $H$ and Frobenius kernel $K$. Then $\Gamma(G)$ is a cograph {\rm(}resp. chordal graph{\rm)} if and only if $\Gamma(K)$ is a cograph {\rm(}resp. a chordal graph{\rm)} and $\Gamma(H)$ is a cograph {\rm(}resp. a chordal graph{\rm)}.
\end{cor}

\begin{remark}
Since $K$ is nilpotent, Corollary~{\rm \ref{Nilp-cog}} gives a criterion for
$\Gamma(K)$ to be a cograph {\rm(}resp. a chordal graph{\rm)}.
\end{remark}

\begin{ques} For which Frobenius groups is the commuting graph a cograph or a chordal graph?
\end{ques}

\subsection{A general structure theorem}\label{General}

The main result of this section is the following theorem.

\begin{theorem}\label{CommonStructure_Cograph}
Let $G$ be a finite group such that $\Gamma(G)$ is a cograph {\rm(}resp.\ a chordal graph{\rm)}, and let $F(G)$ be the Fitting subgroup of $G$.
\begin{enumerate}
\item One of the following statements holds:
  \begin{enumerate}
  \item $C_G(F(G))=Z(F(G))$;
  \item $F(G)$ is of nilpotent class at most~$2$ and $G/F(G)$ is almost simple;
in particular, $G$ has a unique non-abelian composition factor.
  \end{enumerate}
\item If $\Soc(G)$ is non-abelian, then $G$ has a normal subgroup $N=A \times S$, where $A$ is abelian, $S$ is simple and $\Gamma(S)$ is a cograph {\rm(}resp.\ a chordal graph{\rm)}, such that $G/N$ is isomorphic to a subgroup of $\Out(S)$ and, therefore, $G/N$ is solvable.
\end{enumerate}
\end{theorem}

\begin{proof} Let $F^*(G)$ be the generalized Fitting subgroup of $G$.
Let $H_1$ and $H_2$ be distinct elements from $\Comp(G)$. By \cite[31.5]{Asch86}, $H_1$ and $H_2$ commute, therefore $\langle H_1, H_2 \rangle = H_1 \circ H_2$ is a subgroup of $G$. By Lemma~\ref{lem_cp}, $\Gamma(G)$ is neither a cograph nor a chordal graph, a contradiction to Proposition~\ref{cog-subg}. Thus, $|\Comp(G)| \le 1$.

If $|\Comp(G)|=0$, then $F^*(G)=F(G)$, and because we have $C_G(F^*(G))\le F^*(G)$ by \cite[31.13]{Asch86}, statement (a)(i) holds.

Assume that $|\Comp(G)|=1$. Then $G$ has a quasisimple characteristic subgroup $N=E(G)$. Consider $C=C_G(N)$, which is a normal subgroup of $G$. Let $H = \langle N ,C \rangle = N \circ C$. By Proposition~\ref{cog-subg} and Lemma~\ref{lem_cp}, $C$ is nilpotent of nilpotency class at most~$2$, therefore $C \le F(G)$.
On the other hand, $F(G) \le C$ by (\ref{F(G)E(G)}). Thus, $C=F(G)$ and, therefore $G/F(G)$ is a subgroup of $\Aut(N)$ which has a simple normal subgroup $H/C \cong N/Z(N)$.

Now we show that $\Aut(N)$ is an almost simple group and $\Soc(\Aut(N))\cong N/Z(N)$. Indeed, $N/Z(N)\cong\Inn(N)$ is a normal subgroup of $\Aut(N)$ and $N/Z(N)$ is simple. By Proposition~\ref{Wielandt}, we have $$\Aut(N) \cong\Aut(N)/C_{\Aut(N)}(\Inn(N))$$ is a subgroup of $\Aut(\Inn(N))\cong \Aut(N/Z(N))$. So, we have that $\Aut(N)$ is a subgroup of an almost simple subgroup $\Aut(N/Z(N))$ which contains a normal subgroup $\Inn(N)\cong N/Z(N)$. Thus, $\Aut(N)$ is an almost simple group with socle $N/Z(N)$, as required.

So, $G/F(G)$ is isomorphic to a subgroup of $\Aut(N)$ which is almost simple with socle isomorphic to $\Inn(N)$ and contains a subgroup $H/F(G) \cong\Inn(N)$, therefore $G/F(G)$ is almost simple. Thus, statement (a)(ii) holds.

\smallskip

Now we prove statement (b). Assume that $L=\Soc(G)$ is non-abelian. Then $L=A\times S$, where $A$ is abelian and $S$ is the direct product of non-abelian simple groups, moreover, both $A$ and $S$ are normal subgroups of $G$. By Propositions~\ref{cog-subg} and~\ref{zj-thm}, $S$ is simple and $\Gamma(S)$ is a cograph {\rm(}resp. a chordal graph{\rm)}. If $C=C_G(S)$, then $C \unlhd G$ and $C \cap S =1$ since $S$ is simple.
Thus, $N=SC=S \times C$ is a normal subgroup of $G$. By Proposition~\ref{cog-subg}, $\Gamma(N)$ is also a cograph (resp.\ a chordal graph). Using Proposition~\ref{zj-thm}, we conclude that $C$ is abelian.
Moreover, $G/C$ is isomorphic to a subgroup of  $\Aut(S)$ which is almost simple, therefore $G/N$ is isomorphic to a subgroup of $\Out(S)$. Thus, by the Schreier Conjecture (Theorem~\ref{schreier}), $G/N$ is solvable, and therefore statement (b) holds. \qed
\end{proof}

\begin{remark} If $\Gamma(G)$ is a cograph, then in statement~$(b)$ of Theorem~{\rm \ref{CommonStructure_Cograph}} the group $G/N$ is cyclic. This result follows from the classification of finite simple groups whose commuting graphs are cographs, see Theorem~{\rm \ref{Simple-cog}} below.
\end{remark}

The following problem is of interest.

\begin{problem}\label{Classify} Classify finite simple groups whose commuting graphs are chordal.
\end{problem}

\section{Simple groups whose commuting graph is a cograph or a chordal graph}\label{Simple}

In this section, we investigate when the commuting graph of a simple sporadic or alternating group is a cograph or a chordal graph, and then obtain a complete classification of finite simple groups whose commuting graph is a cograph.

\subsection{CA-groups and AC-groups}

A group $G$ is a \emph{CA-group} if the centralizer of every non-identity
element of $G$ is abelian.
Another name is \emph{CT-groups}, since they are the groups for which
the commutation relation on non-identity elements is transitive. (If
$x\sim y\sim z$ in the CA-group $G$, then $x,z\in C_G(y)$, so $x$ and
$z$ commute. The converse is similar.)

\medskip

If $G$ is a CA-group, then the induced subgraph $\Gamma(G)$ on non-identity
elements is a disjoint union of complete graphs. As a result, we see:

\begin{prop}
Let $G$ be a CA-group. Then $\Gamma(G)$ is a cograph, chordal, and perfect.
\end{prop}

CA-groups played an important historical role in developing techniques for
attacking problems on finite groups; they are now well-understood.
From results of Weisner~\cite{weisner}, Brauer, Suzuki and Wall
\cite{BSW} and Suzuki~\cite{suzuki}, we obtain:

\begin{prop}
A finite CA-group $G$ satisfies one of the following:
\begin{enumerate}
\item $G$ is abelian;
\item $G$ is a Frobenius group with abelian Frobenius kernel and cyclic
Frobenius complement; or
\item $G\cong\PSL_2(q)$ for $q$ a power of~$2$.
\end{enumerate}
\label{ca-thm}
\end{prop}

More generally, $G$ is said to be an \emph{AC-group} if $C_G(x)$ is abelian
for every element $x\in G\setminus Z(G)$. This is a weaker condition, though
clearly a group with trivial centre is an AC-group if and only if it is a
CA-group.

AC-groups were investigated by Schmidt~\cite{Sch70} and Rocke~\cite{Ro75}.
(Schmidt investigated groups in which two maximal centralizers intersect just
in $Z(G)$; it is not hard to see that, for finite groups, this is equivalent to
the AC property.) Abdollahi~\emph{et al.}~\cite{AA06} investigated the
connection of this property with the \emph{non-commuting graph}, the
complement of the commuting graph. The classification of the AC-groups is not
as complete as that of CA-groups, but the cited authors have proved a
number of results about these groups.

From our point of view, AC-groups are as good as CA-groups:

\begin{prop}\label{cent-abe}
Let $G$ be an AC-group. Then $\Gamma(G)$ is a cograph and a chordal graph.
\end{prop}

\begin{proof}
By Proposition~\ref{cog-subg}, it suffices to consider the induced subgraph on
the set $G\setminus Z(G)$; and this graph is a disjoint union of complete
graphs. \qed
\label{ac-prop}
\end{proof}

Here is a simple general result which we will use.

\begin{prop}
Let $G$ be a finite group with the property that every non-identity element
of $G/Z(G)$ has cyclic centralizer. Then $G$ is an AC-group.
\label{p-ac-cyc}
\end{prop}

\begin{proof}
Take $a\in G\setminus Z(G)$. Then $Z(G)\le C_G(a)$, and $C_G(a)/Z(G)\le C_G(aZ(G))$ (in
$G/Z(G)$); so $C_G(a)/Z(G)$ is cyclic, whence $C_G(a)$ is abelian. \qed
\end{proof}

\subsection{Generalized dihedral groups and generalized quaternion groups}

In this subsection we show that these groups are AC-groups, so that their
commuting graphs are cographs as well as chordal graphs.

\begin{prop}\label{gdg-thm}
Let $D(A)$ be the generalized dihedral group as presented in
Equation~(\ref{gdg}). Then $D(A)$ is an AC-group. Hence $\Gamma(D(A))$ is a
cograph as well as a chordal graph. In particular, the commuting graph of a
dihedral group is a cograph as well as a chordal graph.
\end{prop}

\begin{proof}
The centre of $D(A)$ is the subgroup $T$ of $A$ consisting of elements of orders $1$ or $2$. If we remove it, then $A\setminus T$ is a clique, and the elements outside $A$ fall into cliques of the form $\langle t,T\rangle\setminus T$ for $t\notin A$; there are no edges between these cliques. So the induced subgraph is a disjoint union of complete graphs. Now Proposition~\ref{cog-subg} gives the
result.\qed
\end{proof}

\begin{prop}
Let $Q_{4m}$ be the generalized quaternion group as presented in
Equation~(\ref{q4m}). Then $Q_{4m}$ is an AC-group. Hence $\Gamma(Q_{4m})$ is a
cograph as well as a chordal graph.
\end{prop}

\begin{proof}
The proof is similar to the preceding theorem. We find that, with $G=Q_{4m}$,
we have $Z(G)=\{e,x^m\}$ (Observation~\ref{obs2}); for
$1\le i\le 2m-1$ and $i\ne m$, we find that
$C_G(x^i)=\langle x\rangle$, while for $0\le i\le 2m-1$ we have
$C_G(x^iy)=\langle x^iy\rangle=\{e,x^iy,x^m,x^{m+i}y\}$. So again the
induced subgraph on $G\setminus Z(G)$ consists of a complete graph of size
$2m-2$ together with $m$ disjoint edges. \qed
\end{proof}

\subsection{Symmetric groups and alternating groups }

\begin{prop}\label{syg-cg}
$\Gamma(S_n)$ is a cograph if and only if $n\le 3$.
\end{prop}

\begin{proof}
Proposition~\ref{cog-subg} shows that it is enough to show that $\Gamma(S_3)$ is a
cograph but $\Gamma(S_4)$ is not (since $S_{n-1}\le S_n$ for any $n$).
By Proposition~\ref{gdg-thm}, we have that $\Gamma(S_3)$ is a cograph as $S_3\cong D_{6}$.
As $S_{n-1}\le S_n$, Proposition~\ref{cog-subg} shows that if $n\le 3$, then $\Gamma(S_n)$ is a cograph. In $\Gamma(S_4)$,
$$
(12)\sim(12)(34)\sim(13)(24)\sim(13)
$$
is an induced path isomorphic to $P_4$.\qed
\end{proof}

\begin{prop}\label{sg-chog}
$\Gamma(S_n)$ is a chordal graph if and only if $n\le 4$.
\end{prop}

\begin{proof}
It is easy to see that
$$
(12)\sim(34)\sim(15)\sim(24)
\sim(35)\sim(12)
$$
is an induced cycle isomorphic to $C_{5}$ in $\Gamma(S_5)$. As in
Proposition~\ref{syg-cg}, it suffices to prove that $\Gamma(S_4)$
is a chordal graph. Note that in
$\Gamma(S_4)$, for any non-central element, the cycle type of this element must be one of
\begin{equation}\label{cho-eq-10}
\{(ij), (ij)(kl), (ijk),(ijkl)\},
\end{equation}
where $\{i,j,k,l\}=\{1,2,3,4\}$.
Moreover, it is easy to see that
\[C_{S_4}((ij))=\{e,(ij),(kl),(ij)(kl)\},
C_{S_4}((ij))=\{e,(ij),(kl),(ij)(kl)\}\]
\begin{equation}\label{cho-eq-11}
C_{S_4}((ijkl))=\langle(ijkl)\rangle
\end{equation}
and
$$
C_{S_4}((ij)(kl))=\{e, (ij), (kl), (ij)(kl), (ik)(jl), (il)(jk), (ikjl), (iljk)\}.
$$
Suppose, for the sake of contradiction, that $\Gamma(S_4)$ contains an induced subgraph isomorphic to $C_n$ with $n\ge 4$, say $\Delta$. Let $a,b,c,d$ be $4$ distinct vertices of $\Delta$.
Then $C_{S_4}(x)$ is non-abelian for any $x\in \{a,b,c,d\}$.
Note that every centralizer in (\ref{cho-eq-11}) is abelian. It follows from (\ref{cho-eq-10}) that every element in $\{a,b,c,d\}$ has
cycle type $(ij)(kl)$, which is impossible because $S_4$ has precisely three elements that are the product of two disjoint cycles. \qed
\end{proof}

\begin{prop}\label{ag-cg}
The following are equivalent{\rm:}

$(a)$ $\Gamma(A_n)$ is a cograph{\rm;}

$(b)$ $\Gamma(A_n)$ is a chordal graph{\rm;}

$(c)$ $n\le 5$.
\end{prop}

\begin{proof}
The group $A_5$ is a CA-group (it is easily checked from~\cite{CC85} that all
its non-abelian  subgroups have trivial centre), so the result for
$n=5$ follows from Proposition~\ref{ca-thm}. Moreover, $A_n\le A_5$ for $n\le 5$,
so the result follows for these groups as well.

On the other hand, in $\Gamma(A_6)$, it is easy to verify that
$$
(12)(56)\sim(12)(34)\sim(13)(24)\sim(13)(56)
$$
is an induced path isomorphic to $P_4$.
Also, it is not hard to verify that
\begin{eqnarray*}
&&(34)(56)\sim(12)(56)\sim(15)(26)\sim(26)(34)\sim(24)(36)\\
&&\sim(15)(24)\sim(12)(45)\sim(36)(45)\sim(34)(56)
\end{eqnarray*}
is an induced cycle isomorphic to $C_8$, so $\Gamma(A_6)$ is not chordal.
Now $A_6\le A_n$ for $n>6$, so $\Gamma(A_n)$ is neither a cograph nor chordal
for $n\ge6$, by Proposition~\ref{cog-subg}.
\qed
\end{proof}

In the commuting graph of $A_6$ with the identity removed, the elements of
orders $3$ and $5$ induce a disjoint union of cliques, while an element of
order~$4$ is joined only to its inverse and square. So induced cycles on
more than three vertices must consist of elements of order $2$. We can describe
the induced subgraph on these elements as follows.

\begin{prop}
Let $\Delta$ be the induced subgraph of the commuting graph of $A_6$ on the set
of involutions. Then $\Delta$ is isomorphic to the line graph of Tutte's
$8$-cage~\cite{Tutte}. In particular, an induced cycle in $\Gamma(A_6)$ with
length greater than~$3$ has length at least~$8$.
\end{prop}

\begin{proof} Tutte's $8$-cage $T$ (the smallest cubic graph with girth~$8$)
can be defined as follows. Vertices are odd involutions in the symmetric group
$S_6$; these are of two types, transpositions and products of three
transpositions. Two vertices are joined if one is a product of three
transpositions, one of which is the other. Further information on this graph
can be found in \cite[p.~209]{BCN}.

There is a bijection between vertices of $\Delta$ and edges of $T$: for
example, the vertex $(12)(34)$ of $\Delta$ corresponds to the edge
$\{(12)(34)(56), (56)\}$ of $T$. A simple check shows that this bijection is
an isomorphism between $\Delta$ and the line graph of $T$.

Now it is easily checked that $T$ has girth~$8$; it is in fact the incidence
graph of the unique generalized quadrangle of order~$2$. It follows that, apart
from triangles coming from three edges through a vertex of $T$, the shortest
cycles in $\Delta$ have length~$8$.

The fact that this holds for the graph $\Gamma(A_6)$ now follows from the
remarks before the Proposition.\qed
\end{proof}

\begin{remark}
However, $\Gamma(A_7)$ has an induced cycle isomorphic to $C_{6}$, such as $(123)\sim(456)\sim(127)\sim(345)\sim(126)\sim(457)\sim(123)$. Also, $\Gamma(A_8)$ has an induced cycle isomorphic to $C_{4}$, such as $(12)(34)\sim(567)\sim(348)\sim(12)(56)\sim(12)(34)$.
\end{remark}

\subsection{Sporadic groups}

In this subsection, we resolve the question for the sporadic simple groups.

\begin{prop}
If $G$ is a sporadic simple group, then $\Gamma(G)$ is neither a cograph nor a chordal graph.
\end{prop}

\begin{proof}
There are $26$ sporadic simple groups.
Observe that $S_5$ is a subgroup of the Mathieu group $M_{11}$ by
using information in the $\mathbb{ATLAS}$ of finite groups \cite{CC85}.
Also, note that the fact that the Mathieu group $M_{11}$ is a subgroup of all the other sporadic simple groups except $M_{22}$, $J_1$, $J_2$, $J_3$, $He$, $Ru$ and $Th$.
Thus, by Proposition~\ref{cog-subg} and
Propositions~\ref{syg-cg} and \ref{sg-chog}, the commuting graph of a sporadic simple group containing $M_{11}$ is neither a cograph nor a chordal graph.

Now by \cite{CC85} again, it is easy to observe that
$A_7\le M_{22}$, $D_6\times D_{10}\le J_1$, $A_4\times A_{5}\le J_2$, $A_6\le J_3$, $S_6\le He$, $A_8\le Ru$, $S_5\le Th$.
Thus, it follows from Proposition~\ref{cog-subg}, Lemma~\ref{zj-cg} and Proposition~\ref{ag-cg} that
the commuting graphs of these $7$ sporadic simple groups are neither cographs nor chordal graphs, as desired. \qed
\end{proof}

\subsection{Simple groups of Lie type whose commuting graph is a cograph}

Now we turn to simple groups of Lie type. In this analysis, we only consider
cographs, not chordal graphs.

\begin{prop}
Let $G=\PSL_2(q)$, with $q\ge4$. Then $\Gamma(G)$ is a cograph if and only if
either $q$ is a power of $2$ or $q=5$.
\end{prop}

Note that $\PSL_2(5)\cong\PSL_2(4)$.

\begin{proof}
The subgroup structure of $\PSL_2(q)$ was determined by Dickson~\cite{Dickson},
and can be found, for example, in~\cite{BHRD}.

Suppose first that $q\ge 4$ is a power of $2$.
Then every non-trivial element of $\PSL_2(q)$ has order $2$ or a divisor of $q-1$ and $q+1$. Also,
the centralizers of involutions of $\PSL_2(q)$ are elementary abelian groups with order $q$, and the centralizers of other non-trivial elements of $\PSL_2(q)$ are cyclic of order $q-1$ or $q+1$. It follows that in this case $\PSL_2(q)$ is a CA-group, and so by Proposition~\ref{ca-thm}, $\Gamma(\PSL_2(q))$ is a cograph.

Suppose next that $q$ is odd at least $5$. Let $\epsilon\in\{+1,-1\}$ be chosen so that $q+\epsilon1$ is divisible by $4$.
Note that $\PSL_2(q)$ contains a Klein group such that
the centralizers of its involutions are dihedral groups $D_{2(q+\epsilon1)}$ which contain disjoint cyclic groups of order $(q+\epsilon)/2$.
If $q>5$, then $(q+\varepsilon)/2>2$. So two involutions together with generators of cyclic groups of order $(q+\epsilon)/2$ in their centralizers induce a path of length~$4$. (The list of subgroups of $G$ shows that these generators do not commute.) So if $q>5$ is odd, then $\Gamma(\PSL_2(q))$ is not a cograph. \qed
\end{proof}

\begin{remark} It is well-known that $\PSL_2(2)\cong S_3$ and $\PSL_2(3)\cong A_4$, therefore their commuting graphs are cographs by Propositions~{\rm\ref{syg-cg}} and~{\rm\ref{ag-cg}}, respectively.
\end{remark}

\begin{prop}\label{cog-PSL}
Let $G$ be either $\PSL_3(q)$ or a group of Lie type of BN-rank at least~$3$.
Then $\Gamma(G)$ is not a cograph.
\end{prop}

\begin{proof}
We begin by showing that, if $G=\SL_3(q)$, then $\Gamma(G)$ is not a cograph.

Suppose first that $q\notin\{2,4\}$. Choose a non-zero element $a\in\GF(q)$
such that $a^3\ne1$, and let $b=a^{-2}$, so that $b\ne a$.

Consider the four elements $g_1,\ldots,g_4$ given by
\[
\left(\begin{array}{ccc}1&1&0\\0&1&0\\0&0&1\end{array}\right),\,
\left(\begin{array}{ccc}a&0&0\\0&a&0\\0&0&b\end{array}\right),\,
\left(\begin{array}{ccc}b&0&0\\0&a&0\\0&0&a\end{array}\right),\,
\left(\begin{array}{ccc}1&0&0\\0&1&1\\0&0&1\end{array}\right).
\]
It is clear that all of them belong to $G$, and that
$g_1\sim g_2\sim g_3\sim g_4$ in $\Gamma(G)$. A short calculation gives the
commutators of the other three pairs; none of them is trivial. So we have
an induced path in $\Gamma(G)$.

But the argument gives more. In fact none of the commutators is a scalar
matrix, so if $Z$ is any subgroup of the centre of $G$ (the group of scalar
matrices), then we also have an induced path in $\Gamma(G/Z)$, so that
$\Gamma(G/Z)$ is not a cograph. In particular, $\Gamma(\PSL_3(q))$ is not a
cograph.

Now we can cover the remaining two cases. Since $\PSL_3(2)\cong\PSL_2(7)$,
the commuting graph of $\PSL_3(2)$ is not a cograph. Also, $\PSL_3(2)$ is a
subgroup of both $\SL_3(4)$ and $\PSL_3(4)$, so their commuting graphs are not
cographs either.

For the remainder of the proof, we use the Levi decomposition of parabolic
subgroups of groups with BN-pairs (see~\cite[Section 8.5]{Carter}).
The Borel subgroup $H$ of such a group is generated by the root subgroups
corresponding to the positive roots with respect to a fundamental system
$\Pi$ of roots. Parabolic subgroups $P_\Phi$ are those which contain $H$; they
are generated by $H$ together with the root subgroups corresponding to the
negatives of a subset $\Phi$ of the negatives of the fundamental group.
According to the Levi decomposition, $P_\Phi$ is the semidirect product
of its unipotent radical by the \emph{Levi factor} $L_\Phi$, a group with
BN-pair corresponding to the roots $\Phi$.
In the covering group, $L_\Phi$ may have a centre, which may or may not be
killed when we factor out the centre of $G$. (For example, $\SL_n(q)$ has a
parabolic subgroup corresponding to the first two roots, generated by all
upper unitriangular matrices together with the lower unitriangular matrices
with non-zero off-diagonal elements in positions $1,2,3$ only. These generate
$\SL_3(q)$; its image in $\PSL_n(q)$ is either $\SL_3(q)$ or $\PSL_3(q)$
depending on the congruence of $n$ and $q$ mod~$3$.)

Now the Coxeter--Dynkin diagram of any simple group $G$ of Lie type with
BN-pair of rank at  least $3$ contains a single edge, corresponding to a root
system of type $A_2$; so $G$ has a parabolic
subgroup whose Levi factor is a central quotient of $\SL_3(q)$, as in the
example above. From Proposition~\ref{cog-subg},
we conclude that $\Gamma(G)$ is not a cograph. \qed
\end{proof}

\begin{prop}\label{Thm_SU_3}
Let $G$ be a central quotient of $\SU_3(q)$, with $q>2$. Then $\Gamma(G)$ is
not a cograph.
\end{prop}

\begin{proof}
Suppose that $q>2$.
We take the Hermitian form $H(x_1,x_2,x_3)=x_1^{q+1}+x_2^{q+1}+x_3^{q+1}$.
A diagonal matrix with diagonal entries $(a,b,c)$ preserves this form if and
only if
\[a^{q+1}=b^{q+1}=c^{q+1}=1,\]
and belongs to $\SU_3(q)$ if additionally $abc=1$.
So the set of such matrices forms a subgroup of $\SU_3(q)$ isomorphic to
$\mathbb{Z}_{q+1}^2$, and its quotient in $\PSU_3(q)$ is
$\mathbb{Z}_{q+1}\times \mathbb{Z}_{(q+1)/\gcd(q+1,3)}$. We observe that $(-1)^{q+1}=1$ (this
is clear if $q$ is even and holds if $q$ is odd since then $(-1)^2=1$).

Let $a\in GF(q^2)$ satisfy $a^{q+1}=1$, $a^3\ne1$, and let $b=a^{-2}$, so
that $a\ne b$. Now consider the four matrices $g_1$, \dots, $g_4$ given by
\[
\left(\begin{array}{ccc}0&1&0\\1&0&0\\0&0&-1\end{array}\right),\,
\left(\begin{array}{ccc}a&0&0\\0&a&0\\0&0&b\end{array}\right),\,
\left(\begin{array}{ccc}b&0&0\\0&a&0\\0&0&a\end{array}\right),\,
\left(\begin{array}{ccc}-1&0&0\\0&0&1\\0&1&0\end{array}\right).
\]
It is readily checked that all four matrices belong to $\SU_3(q)$, and that
$g_2$ commutes with $g_1$ and $g_3$ while $g_3$ commutes with $g_4$. The
other pairs fail to commute; indeed in each case $g_ig_j$ is not a scalar
multiple of $g_jg_i$. So the four matrices induce a $P_4$ subgraph of the
commuting graph, which remains as a $P_4$ subgraph when the scalar matrices
in $\SU_3(q)$ are factored out. \qed
\end{proof}

\medskip

Among the groups of Lie type, this leaves only those of ranks $1$ or $2$,
and we have already handled $\PSL_2(q)$, $\PSL_3(q)$, and $\PSU_3(q)$ for $q>2$. The remaining groups
are:
\begin{enumerate}\itemsep0pt
\item $\PSp_4(q)$;
\item $\PSU_n(q)$ for $n\in\{4,5\}$;
\item $G_2(q)$, $\Sz(q)={}^2B_2(q)$, $R_1(q)={}^2G_2(q)$, $R_2(q)={}^2F_4(q)$, and
${}^3D_4(q)$.
\end{enumerate}

We have omitted the orthogonal groups here. For these satisfy
\begin{enumerate}
\item $\POm_3(q)\cong\PSL_2(q)$;
\item $\POm^+_4(q)\cong\PSL_2(q)\times\PSL_2(q)$ and $\POm^-_4(q)\cong\PSL_2(q^2)$;
\item $\POm_5(q)\cong\PSp_4(q)$;
\item $\POm^+_6(q)\cong\PSL_4(q)$;
\item $\POm^-_6(q)\cong\PSU_4(q)$
\end{enumerate}
(see the $\mathbb{ATLAS}$ of Finite Groups~\cite{CC85}). So (a), (b) and (d)
are covered by our earlier results, and (c) and (e) will be treated as symplectic or unitary groups.

Also $G_2(q)$ contains a subgroup $\PSL_3(q)$, and ${}^3D_4(q)$
contains $G_2(q)$ \cite{Cooperstein,Kleidman,Kleidman2}, so their commuting graphs are not cographs, by Proposition~\ref{cog-subg} and Proposition~\ref{cog-PSL}.

We now finish off the classical groups.

The group $\PSp_4(q)$ contains the central product of two copies of
$\SL_2(q)$. (To see this, note that the $4$-dimensional vector space can be
written as the orthogonal direct sum of two $2$-dimensional non-degenerate subspaces; so
$\Sp_4(q)$ contains $\SL_2(q)\times\SL_2(q)$. If $q$ is odd, take the
quotient by $\{\pm I\}$ to obtain a central product.) So its commuting graph
is not a cograph, by Lemma~\ref{lem_cp}.

For unitary groups, the $\PSU_3(2)$ is not simple; so for $q=2$ it remains to deal with $\PSU_4(2)$ and $\PSU_5(2)$. The first is isomorphic to $\PSp_4(3)$~\cite{CC85}, which we have dealt with. The second contains $\PSL_2(11)$, see~\cite{CC85}. For $q>2$, by Proposition~\ref{Thm_SU_3} central quotients of $\SU_3(q)$ have
commuting graphs which are not cographs. For $n=5$, $\PSU_n(q)$ contains such
a central quotient, obtained as above by writing the $5$-dimensional unitary
space as the sum of non-degenerate subspaces of dimensions $2$ and $3$. For
$n=4$, the space is the sum of two non-degenerate $2$-dimensional subspaces,
so $\PSU_4(q)$ contains a central product of two copies of $\SU_2(q)$.  So the
commuting graphs of these groups are not cographs either.

This concludes dealing with the classical groups.

\medskip

We show next that the commuting graphs of the Suzuki groups are cographs.

\begin{prop}
For each $q$, the commuting graph of $\Sz(q)={}^2B_2(q)$ is a cograph.
\end{prop}

\begin{proof}
It is well-known that ${}^2B_2(2)\cong 5:4$ is a Frobenius groups with cyclic kernel and cyclic complement. Therefore, by Proposition~\ref{Frobenius}, $\Gamma({}^2B_2(2))$ is a cograph. Thus, we can assume that $q>2$.

The information used in the next two paragraphs is taken from Suzuki's original
paper~\cite{suzuki2}; it can also be found in \cite[Section 4.2]{Wilson-book}.

It suffices to show that the graph obtained by removing the identity is a
cograph. Now ${}^2B_2(q)$ has maximal cyclic subgroups of orders $q-1$ and
$q\pm\sqrt{2q}+1$, and every element of odd order lies in one of these; so
the non-identity elements of odd order carry a disjoint union of complete
graphs, with no edges to the remaining elements.
The remaining elements have order~$2$ or $4$, and since ${}^2B_2(q)$ is a
\emph{CIT-group} (one in which the centralizer of any involution is a
$2$-group), the remaining vertices lie in a disjoint union of copies of the
commuting graph on $P\setminus\{e\}$, where $P$ is a Sylow $2$-subgroup. Now
$|P|=q^2$ and $|Z(P)|=q$; elements of $Z(P)$ are joined to every element of $P$
so we may remove them. The remaining elements of $P$ have order $4$. If $g$
is one of these, then $|C_P(g)|=2q$ and $C_P(g)$ consists of two cosets of
$Z(P)$, so it is abelian; thus the elements of order $4$ fall into $q-1$
sets of size $q$, each set a complete graph with no edges between different
sets. Thus the graph is indeed a cograph.
\qed
\end{proof}

Next, the small Ree group $R_1(q)={}^2G_2(q)$, with $q$ an odd power of $3$, is not
simple for $q=3$; for larger $q$, it contains $\PSL_2(q)$ (the centralizer
of an involution in ${}^2G_2(q)$ is $\mathbb{Z}_2\times\PSL_2(q)$ \cite{Kleidman2}).

The \emph{Tits group} ${}^2F_4(2)'$ contains a subgroup isomorphic to
$\PSL_2(25)$ (this can be found in the $\mathbb{ATLAS}$ \cite{CC85}); so its
commuting graph is not a cograph. The group ${}^2F_4(2)'$ is a subgroup of the
large Ree group $R_2(2^d)={}^2F_4(2^d)$ for all odd $d$ \cite{Malle}.

\medskip

To summarize our results, combined with the earlier results on alternating
and sporadic groups, and noting that $A_5 \cong\PSL_2(4)$:

\begin{theorem}\label{Simple-cog}
Let $G$ be a non-abelian finite simple group. Then the commuting graph of $G$
is a cograph if and only if $G$ is isomorphic to $\PSL_2(q)$, with $q$ a power
of $2$ and $q>2$, or to $\Sz(q)={}^2B_2(q)$, with $q$ an odd power of $2$ and $q>2$.
\end{theorem}

\section{Appendix}\label{appendix}

In this section, we consider small groups whose commuting graphs are cographs or chordal graphs.

\begin{lemma}\label{sg-ac}
If $|G/Z(G)|=pqr$ with primes $p,q,r$ {\rm(}not necessary distinct{\rm)}, then $G$ is an
AC-group.
\end{lemma}

\begin{proof}
Let $a\in G\setminus Z(G)$. It suffices to prove that $C_G(a)$ is abelian. Note that $Z(G)\subseteq Z(C_G(a))\subseteq C_G(a)<G$. Since $|G/Z(G)|=pqr$, we have that $|C_G(a)/Z(G)|<pqr$ and $|C_G(a)/Z(G)|$ is a divisor of $pqr$. By the argument of Proposition~\ref{p-ac-cyc}, it suffices to consider
the case where $C_G(a)/Z(G)$ is non-cyclic and has order $pq$.

Then $Z(G)a$ has order $p$ or $q$. Thus, $C_G(a)/Z(G)$ has an element, say $Z(G)b$, such that $C_G(a)/Z(G)=\langle Z(G)b,Z(G)a\rangle$. Note that $[b,a]=1$. Hence, $C_G(a)=\langle b,a,Z(G)\rangle$ is an abelian group, as desired. \qed
\end{proof}

A similar but easier proof shows the following.

\begin{cor}\label{sg-co1}
If $|G/Z(G)|=pq$ or $p$ where $p,q$ are primes, then $G$ is an AC-group. In particular, If $|G|=pq$ or $pqr$ with primes $p,q,r$, then $G$ is either an abelian group or an AC-group.
\end{cor}

Note that these primes $p,q,r$ in Lemma~\ref{sg-ac} and Corollary~\ref{sg-co1} are not necessarily distinct.
Now combining Corollary~\ref{sg-co1},
Lemma~\ref{sg-ac} and Proposition~\ref{cent-abe}, we have the following result.

\begin{lemma}\label{sg-pdc}
$\Gamma(G)$ is a cograph as well as a chordal graph if one of the following holds{\rm:}
\begin{enumerate}
\item $G$ is abelian;
\item $|G|=pqr$ or $pq$, where $p,q,r$ are primes {\rm(}not necessarily distinct{\rm);}
\item $|G/Z(G)|=p$, $pq$ or $pqr$, where $p,q,r$ are primes {\rm(}not necessary distinct{\rm)}.
\end{enumerate}
\end{lemma}

\begin{prop}\label{mini-cog}
If $G$ is a group of minimal order such that $\Gamma(G)$ is not a cograph, then $|G|=24$. In particular, for $|G|\le 24$, $\Gamma(G)$ is a cograph if and only if $G\ncong S_4$.
\end{prop}

\begin{proof}
Proposition~\ref{syg-cg} implies that $\Gamma(S_4)$ is not a cograph. Thus, it suffices to show that if $|G|\le 24$, then $\Gamma(G)$ is a cograph if and only if $G\ncong S_4$. Now suppose that $|G|\le 24$. It follows from Lemma~\ref{sg-pdc} that if $|G|\notin \{16,24\}$, then $\Gamma(G)$ is a cograph.
If $|G|=16$, since a $2$-group has non-trivial centre, we have that $|G/Z(G)|=2^k$ with $k\le 3$, and so Lemma~\ref{sg-pdc}(c) implies that $\Gamma(G)$ is a cograph. Finally, we suppose that $|G|=24$. Note that the fact that if $G\ncong S_4$, then $G$ has non-trivial centre (cf. Burnside \cite{Bur11}). So, by  Lemma~\ref{sg-pdc}(c), we have that if $G\ncong S_4$, then $\Gamma(G)$ is a cograph, as desired. \qed
\end{proof}

In this paper, we identify a finite group by the unique identifier of this finite group in {\sc SmallGroups} library  in the \textsf{GAP} system \cite{gap}.
To be specific,
the $m$-th group of order $n$ is identified as \texttt{SmallGroup}$[n,m]$, where $[n,m]$ is called the GAP ID of the group.

By Lemma~\ref{zj-cg}, $\Gamma(S_3\times S_3)$ is not a cograph.
One can use the \textsf{GAP} system \cite{gap} to classify all finite groups $G$ satisfying $|G|\le 36$ and $\Gamma(G)$ is not a cograph; these groups are presented in Table \ref{tab1}.
In particular, by Lemma~\ref{zj-cg}, $\Gamma(S_3\times S_3)$ is not a cograph.

\begin{table}[htbp]\centering
\caption{All groups $G$ such that $\Gamma(G)$ is not a cograph and $|G|\le 36$  \label{tab1}}
\vskip2mm
{\tabcolsep=10pt
\begin{tabular}{|c|c|}
\hline
GAP ID    &  Structure                                  \\  \hline
$[24,12]$   & $S_4$                           \\
$[32,6]$  & $(\mathbb{Z}_2\times \mathbb{Z}_2\times \mathbb{Z}_2)\rtimes \mathbb{Z}_4$   \\
$[32,7]$  & $(\mathbb{Z}_8 \rtimes \mathbb{Z}_2)\rtimes \mathbb{Z}_2$                                \\
$[32,8]$  & $(\mathbb{Z}_2\times \mathbb{Z}_2)_{\cdot}(\mathbb{Z}_4\times \mathbb{Z}_2)$                               \\
$[32,43]$  & $\mathbb{Z}_8 \rtimes (\mathbb{Z}_2\times \mathbb{Z}_2)$                                 \\
$[32,44]$  & $(\mathbb{Z}_2\times Q_8)\rtimes \mathbb{Z}_2$                 \\
$[32,49]$  & $(\mathbb{Z}_2\times\mathbb{Z}_2\times\mathbb{Z}_2)\rtimes(\mathbb{Z}_2\times\mathbb{Z}_2)$  \\
$[32,50]$  & $(\mathbb{Z}_2\times Q_8)\rtimes \mathbb{Z}_2$   \\
$[36,10]$  & $S_3\times S_3$                                          \\
 \hline
\end{tabular}}
\end{table}

For the groups of order $32$, there are precisely two extra-special $2$-groups $D_8\circ D_8$ and $D_8\circ Q_8$ as follows:
\begin{eqnarray}\label{g-1}
  D_8\circ D_8&=&\langle x,y,z,w|x^2=y^2=z^2=w^2=[x,y]^2=[z,w]^2=1, \\
   &&[x,z]=[z,y]=[y,w]=[w,x]=1 \rangle
\end{eqnarray}
where $[x,y]=[z,w]$ is an involution belonging to $Z(D_8\circ D_8)$;
\begin{eqnarray}\label{g-2}
  D_8\circ Q_8&=&\langle x,y,z,w|x^2=y^2=z^4=w^4=[x,y]^2=[z,w]^2=1, \\
   &&[x,z]=[z,y]=[y,w]=[w,x]=1 \rangle
\end{eqnarray}
where $z^2=w^2=[x,y]=[z,w]$ is an involution belonging to $Z(D_8\circ Q_8)$.

\begin{prop}
If $G$ is a group of minimal order such that $\Gamma(G)$ is not a chordal graph, then $|G|=32$. In particular, if
$G\cong D_8\circ D_8$ or $D_8\circ Q_8$, then $\Gamma(G)$ is not chordal.
\end{prop}

\begin{proof}
Suppose that $G\cong D_8\circ D_8$ or $D_8\circ Q_8$. Then by \eqref{g-1} and
\eqref{g-2}, it is easy to see that the induced subgraph of
$\Gamma(G)$ by the set $\{x,y,z,w\}$
is isomorphic to $C_4$, which implies that $\Gamma(G)$ is not chordal.

Now, it suffices to show that if $|G|<32$, then $\Gamma(G)$ is a chordal graph. Suppose that $|G|\le 31$.
Note that a $p$-group has non-trivial centre where $p$ is a prime. From Lemma~\ref{sg-pdc}, it follows that if $|G|\ne 24$, then $\Gamma(G)$ is a chordal graph. Now consider $|G|= 24$.
If $G\cong S_4$, then by Proposition~\ref{sg-chog}, $\Gamma(G)$ is chordal. If $G\ncong S_4$, then by Lemma~\ref{sg-pdc} and the proof of Proposition~\ref{mini-cog}, $\Gamma(G)$ is chordal, as desired. \qed
\end{proof}

We have taken these computations further and determined, using \textsf{GAP} with
the package \texttt{GRAPE}~\cite{grape}, the groups of order up to $128$ whose
commuting graphs are not cographs. The numbers of groups are given in
Table~\ref{t:small}.

\begin{table}[htbp]
\caption{\label{t:small}Small groups whose commuting graph is not a cograph}
\[\tabcolsep=10pt
\begin{array}{|r|r|r|r|r|r|}
\hline
\hbox{Order} & \hbox{Groups}& \hbox{Order} & \hbox{Groups} & \hbox{Order} & \hbox{Groups}\\\hline
24 & 1 & 32 & 7 & 36 & 1 \\

48 & 10 & 54 & 2 & 60 & 2 \\

64 & 115 & 72 & 11 & 80 & 12 \\

84 & 1 & 96 & 112 & 100 & 2 \\

108 & 10 & 112 & 8 & 120 & 15 \\

126 & 2 & 128 & 1539 & & \\ \hline
\end{array}
\]
\end{table}

\paragraph{Acknowledgements}
The authors of this paper contributed equally to this work, they are ordered with respect to alphabet ordering in English.  The authors are grateful to the referee for thoughtful comments and suggestions that have helped improve our presentation.

This work was initiated during the online series of Research Seminars on
``Groups and Graphs'' in March--August 2021, run by Ambat Vijayakumar and
Aparna Lakshmanan, Cochin Univ.\ of Science and Technology.

The second author acknowledges the Isaac Newton Institute for Mathematical
Sciences, Cambridge, for support and hospitality during the programme
\textit{Groups, representations and applications: new perspectives}
(supported by \mbox{EPSRC} grant no.\ EP/R014604/1), where he held a Simons
Fellowship.

The third author was supported by National Natural Science Foundation of China (Grant Nos. 11801441 and 12326333) and  Shaanxi Fundamental Science Research Project for Mathematics and Physics (Grant No. 22JSQ024).

The fourth author is thankful to Prof.\ Andrey~V.\ Vasil'ev for his helpful comments which improved this text. Also the fourth author started thinking on this project during her visit to the Steklov Institute of Mathematics, and she is very thankful to colleagues from this institution for hospitality and moral support.

We are also grateful to Eamonn O'Brien for resolving one of the open questions
in the paper, finding nilpotent groups of class~$2$ whose commuting graphs are
isomorphic but which are not isoclinic.

\end{document}